\documentclass[11pt,a4paper]{article}

\usepackage[utf8]{inputenc}
\usepackage[T1]{fontenc}
\usepackage{microtype}
\usepackage[english]{babel}
\usepackage{amsmath,amssymb,amsthm,mathtools}
\IfFileExists{newtxtext.sty}{%
  \usepackage{newtxtext,newtxmath}%
}{%
  \usepackage{mathptmx}%
}

\usepackage{enumitem}
\usepackage{geometry}
\usepackage[dvipsnames]{xcolor}
\usepackage{hyperref}
\usepackage{fancyhdr}
\definecolor{InternalLink}{HTML}{0066CC}
\definecolor{CitationLink}{HTML}{0066CC}
\definecolor{ExternalLink}{HTML}{0066CC}
\definecolor{ORCIDGreen}{HTML}{A6CE39}
\hypersetup{
  colorlinks=true,
  linkcolor=InternalLink,
  citecolor=CitationLink,
  urlcolor=ExternalLink,
  bookmarksopen=true,
  bookmarksnumbered=true,
  pdfauthor={Daniel Baratta},
  pdftitle={Overdetermined Singular Problems and Fractional Torsion},
  pdfsubject={Overdetermined singular problems and fractional torsion},
  pdfkeywords={fractional Laplacian, overdetermined problem, isolated singularity, moving planes, Green function}
}
\IfFileExists{orcidlink.sty}{%
  \usepackage{orcidlink}
  \newcommand{\AuthorORCID}{\orcidlink{0009-0003-5499-4047}}
}{%
  \IfFileExists{tikz.sty}{%
    \usepackage{tikz}
    \newcommand{\AuthorORCID}{%
      \href{https://orcid.org/0009-0003-5499-4047}{%
        \tikz[baseline=-0.55ex]
        \node[circle,fill=ORCIDGreen,inner sep=0.27ex,
        text=white,font=\sffamily\bfseries\tiny]{iD};}}
  }{%
    \newcommand{\AuthorORCID}{%
      \href{https://orcid.org/0009-0003-5499-4047}{%
        \textcolor{ORCIDGreen}{\textsf{\bfseries iD}}}}
  }
}

\fancypagestyle{plain}{%
  \fancyhf{}
  \fancyfoot[C]{\thepage}
  }

\renewenvironment{abstract}{%
  \par\noindent\rule{\textwidth}{0.4pt}\par
  \vspace{0.8ex}
  \noindent\textbf{Abstract}\par\smallskip
  \small
}{\par}

\newtheorem{theorem}{Theorem}[section]
\newtheorem{corollary}[theorem]{Corollary}
\newtheorem{proposition}[theorem]{Proposition}
\newtheorem{lemma}[theorem]{Lemma}
\newtheorem{remark}[theorem]{Remark}
\newtheorem{definition}[theorem]{Definition}
\newtheorem{assumption}[theorem]{Assumption}
\numberwithin{equation}{section}

\newcommand{\R}{\mathbb R}

\newcommand{\dd}{\,d}
\newcommand{\PV}{\operatorname{P.V.}}
\newcommand{\dist}{\operatorname{dist}}
\newcommand{\supp}{\operatorname{supp}}
\newcommand{\Lip}{\operatorname{Lip}}
\newcommand{\loc}{\mathrm{loc}}

\title{\normalfont\bfseries Overdetermined Singular Problems and Fractional Torsion}
\author{{\Large Daniel Baratta\,\AuthorORCID}%
\thanks{Corresponding author. E-mail address:
\href{mailto:daniel.baratta@unical.it}{\texttt{daniel.baratta@unical.it}}.}\\[-0.15em]
{\normalsize\itshape Dipartimento di Matematica e Informatica, UNICAL, Ponte Pietro Bucci 31B,}\\[-0.15em]
{\normalsize\itshape 87036 Arcavacata di Rende, Cosenza, Italy}}
\date{}

\begin{document}

\maketitle

\begin{abstract}
We study a Serrin-type overdetermined problem for the fractional Laplacian in
a bounded open set with an isolated non-removable interior singularity. For
positive weak solutions of
\[
\begin{cases}
(-\Delta)^s u=f(u) & \text{in }\Omega\setminus\{0\},\\
u=0 & \text{in }\R^N\setminus\Omega,\\
(\partial_\eta)^s u=c & \text{on }\partial\Omega,
\end{cases}
\]
with \(0<s<1\), \(N>2s\) and \(c<0\), we isolate a first-order tangential
cancellation of the boundary quotient \(u/\delta^s\) that matches the
first-order cancellation needed to close the fractional corner argument and is
weaker than requiring
\(u/\delta^s\in C^1\) in a full boundary neighborhood. Under this condition,
\(\Omega\) is a ball centered at the singular point and \(u\) is radial and
strictly decreasing in the radial variable. No pointwise blow-up rate at the
pole is assumed. If \(\Omega\) is of class \(C^{2,\alpha}\), the cancellation is
automatic when \(s>1/2\), and for every \(s\in(0,1)\) when \(f\) is constant
near zero. In particular, in the torsion case \(f\equiv1\),
\[
u(x)=\tau_R(x)+kG_R(x,0),
\]
where \(\tau_R\) and \(G_R\) are respectively the fractional torsion function
and the Green function of the ball.
\end{abstract}



\medskip
\noindent\rule{\textwidth}{0.4pt}

\section{Introduction and main results}
\label{sec:introduction}

Overdetermined boundary value problems are a classical source of elliptic rigidity. Serrin's symmetry theorem \cite{Serrin} states that if a bounded smooth domain \(\Omega\subset\R^N\) admits a positive solution of
\[
\begin{cases}
-\Delta u=1 & \text{in }\Omega,\\
u=0 & \text{on }\partial\Omega,\\
\partial_\nu u=\text{constant} & \text{on }\partial\Omega,
\end{cases}
\]
then \(\Omega\) is a ball and \(u\) is radially symmetric. Serrin's proof uses moving planes, the maximum principle, Hopf's boundary lemma and a corner lemma.

Singular overdetermined problems have local precedents. Alessandrini and
Rosset proved radial symmetry for degenerate quasilinear equations with a
constant Neumann datum \cite{AlessandriniRosset}, Enciso and Peralta-Salas
obtained rigidity for \(p\)-harmonic functions in punctured domains
\cite{EncisoPeraltaSalas}, and Agostiniani and Magnanini studied
overdetermined data for a Green function \cite{AgostinianiMagnanini}.
Sciunzi \cite{SciunziMovingPlaneSingular} initiated this line of moving-plane
results for local semilinear equations with singular solutions, proving
symmetry and monotonicity under suitable assumptions on the singular set.
Esposito, Farina and Sciunzi \cite{EspositoFarinaSciunzi} subsequently
extended the framework to singular sets of zero \(2\)-capacity, also allowing
possibly singular nonlinearities; see also Esposito, Montoro and Sciunzi for
the quasilinear setting \cite{EspositoMontoroSciunzi}. More recently,
Esposito, Sciunzi and Soave
proved a singular semilinear counterpart of Serrin's theorem
\cite{EspositoSciunziSoave}: a non-removable pole together with the
overdetermined condition forces the domain to be a ball centered at the pole,
and in the torsion case they classify the regular and singular parts. Their
theorem is the direct local prototype for the present problem.

The fractional tools needed to address this question originate in the work of
Fall and Jarohs, who proved an analogue of Serrin's theorem for the restricted
fractional Laplacian \cite{FallJarohs}. They considered overdetermined problems
of the form
\[
\begin{cases}
(-\Delta)^s u=f(u) & \text{in }\Omega,\\
u=0 & \text{in }\R^N\setminus\Omega,\\
(\partial_\eta)^s u=c & \text{on }\partial\Omega,
\end{cases}
\]
where \((\partial_\eta)^s\) is the fractional normal derivative. Their proof
adapts the moving-plane method to nonlocal equations through maximum
principles for antisymmetric functions, together with a fractional Hopf lemma
and a fractional corner lemma. These tools will play a central role in this
paper. An early nonlocal extension of this moving-plane approach was developed
by Barrios, Montoro and Sciunzi \cite{BarriosMontoroSciunzi} for the
fractional Laplacian with a Hardy potential in bounded domains. All citations below to numbered statements in
Fall--Jarohs refer to the revised arXiv version \cite{FallJarohs}.

The fractional Serrin theory has subsequently developed in several directions.
Exterior and annular configurations were studied by Li-Li and by Soave--Valdinoci \cite{LiLiOverdetermined,SoaveValdinoci}. Soave--Valdinoci also
considered a punctured whole-space problem for bounded weak solutions,
with a prescribed finite value at the puncture and without an overdetermined
fractional normal derivative condition. Biswas and Jarohs treated more general
nonlocal operators \cite{BiswasJarohs}. Parallel-surface conditions, exterior
or annular capacity problems, and quantitative stability were investigated in
\cite{CiraoloDipierroPoggesiPollastroValdinoci,CiraoloPollastro,DipierroPoggesiThompsonValdinoci}.
Recent developments also include higher boundary regularity relevant to
nonlocal overdetermined problems \cite{BarriosRosOtonWeidner} and a different
nonlocal Neumann condition prescribed on an exterior parallel surface
\cite{GattiScheuerWeth}. We are not aware of a result in this literature for a
bounded restricted-fractional Dirichlet problem in which the solution itself
has a non-removable isolated interior singularity and a Serrin-type fractional
normal derivative is prescribed constantly on \(\partial\Omega\).

The purpose of the present paper is to develop a nonlocal counterpart of the
singular rigidity result of Esposito--Sciunzi--Soave \cite{EspositoSciunziSoave}.
The proof builds on the antisymmetric maximum, Hopf and corner principles of
Fall--Jarohs \cite{FallJarohs} and on reflection-compatible zero-capacity
cutoffs in the spirit of Montoro--Punzo--Sciunzi \cite{MontoroPunzoSciunzi}.
Two points require additional treatment. \textup{(i)} A symmetry position reached
before the moving plane meets the pole is excluded through the removability
statement of Lemma~\ref{lem:symmetry-removability}, adapting the local insight
of \cite{EspositoSciunziSoave} to the weak fractional setting. \textup{(ii)} At
an orthogonal boundary contact, the Fall--Jarohs corner estimate has order
\(t^{1+s}\); matching this order requires a first-order tangential cancellation
of \(u/\delta^s\), rather than merely H\"older continuity.

We therefore consider a fractional overdetermined problem with an isolated
non-removable singularity. More precisely, we consider
\begin{equation}\label{eq:main-problem}
\begin{cases}
(-\Delta)^s u=f(u) & \text{in }\Omega\setminus\{0\},\\
u>0 & \text{in }\Omega\setminus\{0\},\\
u=0 & \text{in }\R^N\setminus\Omega,\\
(\partial_\eta)^s u=c & \text{on }\partial\Omega,
\end{cases}
\end{equation}
where \(0<s<1\), \(N>2s\), \(c<0\), \(\Omega\subset\R^N\) is a bounded
open set of class \(C^2\) with \(0\in\Omega\), and
\(f\in\Lip_{\loc}(\R_+)\) with \(\mathbb R_+:=[0,+\infty)\). The assumption \(N>2s\) is natural in the
singular fractional setting. In this range the fundamental singularity has
the power-type behavior \(|x|^{2s-N}\), and a point has zero fractional
\(s\)-capacity. The latter fact is essential for the capacity-cutoff argument
that removes the original and reflected poles from the moving-plane tests.
The notion of non-removability is given in Definition~\ref{def:removable-singularity}; in particular, no pointwise asymptotic or
prescribed blow-up rate at the pole is assumed.

The comparison argument follows the zero-capacity strategy of Montoro, Punzo
and Sciunzi \cite{MontoroPunzoSciunzi}. We regularize
\(w_\lambda=u_\lambda-u\) by a reflection-invariant product cutoff removing
\(0\) and \(0^\lambda\). After restriction to the moving half-space this
produces the weighted negative-part test used below, with a capacity error that
tends to zero independently of the singular behavior of \(u\) near the two poles.

To define the overdetermined datum, we use the boundary regularity theory of
Ros-Oton and Serra \cite{RosOtonSerra}. Since the pole is strictly interior, a
cutoff argument controlling the nonlocal commutator and tail terms yields
H\"older continuity of \(u/\delta^s\) near \(\partial\Omega\).

The only point at which a finer boundary behavior is needed is the exclusion
of the orthogonality alternative in the moving-plane argument.  To formulate
the exact cancellation used there, let
\[
\psi:=\frac{u}{\delta^s},
\qquad
\nu_Q:=\text{the inward unit normal to }\partial\Omega\text{ at }Q.
\]
We introduce the following condition:
\begin{equation}\label{eq:TC}
\tag{TC}
\psi\bigl(Q+t(\nu_Q+\tau)\bigr)
-\psi\bigl(Q+t(\nu_Q-\tau)\bigr)
=o(t)
\qquad\text{as }t\to0^+,
\end{equation}
for every \(Q\in\partial\Omega\) and every unit vector
\(\tau\in T_Q\partial\Omega\). The two displayed points belong to
\(\Omega\) for all sufficiently small \(t>0\). When \(N=1\), the tangent
space contains no unit vector, so the condition is vacuous.

\paragraph{Role of the tangential cancellation condition.}
At an orthogonal contact, the fractional corner lemma gives a lower bound of
order \(t^{1+s}\). Since \(u=\delta^s\psi\) and the boundary factor
\(\delta^s\) contributes order \(t^s\), the reflected values of \(\psi\)
must differ by \(o(t)\) to obtain a strictly smaller upper bound. A merely
\(O(t)\) difference would only give order \(t^{1+s}\). Thus \eqref{eq:TC}
matches the order needed in the corner argument; the full calculation is given
in Lemma~\ref{lem:no-corner}.

Full \(C^1\)-regularity of \(\psi\) implies \eqref{eq:TC}, because the
overdetermined datum gives the constant trace \(\psi=-c\) on
\(\partial\Omega\); generic \(C^\alpha\)-regularity with \(\alpha<1\) does
not provide the required order. Thus \eqref{eq:TC} is weaker than the quotient
regularity assumed in the revised Fall--Jarohs theorem
\cite[Remark 1.3]{FallJarohs}. We do not claim necessity or optimality with
respect to other symmetry methods. The higher boundary estimates of
Abatangelo--Ros-Oton \cite{AbatangeloRosOton} make it automatic in the regimes
stated below.

The weakening from full \(C^1\)-regularity to \eqref{eq:TC} is realized
within the class of solutions considered here, and not only at the level of
abstract boundary profiles. Indeed, Appendix~\ref{app:strict-TC-example}
constructs, for every \(s\in(0,1/2)\) with \(N=2\), a positive radial solution of
\[
(-\Delta)^s u=-u\quad\text{in }B_1\setminus\{0\},
\qquad u=0\quad\text{in }\mathbb R^2\setminus B_1,
\]
with a non-removable pole at the origin and constant fractional normal
derivative. Radiality makes \eqref{eq:TC} hold with exact cancellation, while
\[
\frac{u}{\delta^s}
=A_s-B_s\delta^{2s}+o(\delta^{2s})
\qquad\text{as }\delta\to0^+,
\qquad A_s,B_s>0.
\]
Since \(2s<1\), this yields \(u/\delta^s\notin C^1\) up to the boundary. Thus
\eqref{eq:TC} is strictly weaker than full \(C^1\)-regularity even for the
fixed linear nonlinearity \(f(t)=-t\).

The new boundary ingredient used here is therefore the first-order
cancellation \eqref{eq:TC}; the capacity cutoff and antisymmetric boundary
principles remain those of the cited literature. Once symmetry has been
established, the singular part in the torsion case is classified by the
distributional fractional B\^ocher theorem of Li, Wu and Xu
\cite[Theorem 4]{LiWuXu}. The explicit identification then uses the Boggio
representations of the torsion and Green functions in the ball; see
Abatangelo, Jarohs and Salda\~na \cite{AbatangeloJarohsSaldana} and Bucur
\cite{Bucur}.

Our first main result isolates the boundary assumption used in every
range \(0<s<1\):

\begin{theorem}\label{thm:main}
Let \(0<s<1\) and \(N>2s\). Let \(\Omega\subset\R^N\) be a bounded
open set of class \(C^2\) such that \(0\in\Omega\). Let
\(f\in\Lip_{\loc}(\R_+)\), and let \(u\) be a positive solution of
\eqref{eq:main-problem} with a non-removable singularity at the origin.
More precisely, assume that
\[
u\in W^{s,2}_{\loc}(\mathbb R^N\setminus\{0\})
\cap L^1_s(\mathbb R^N)
\cap C(\mathbb R^N\setminus\{0\})
\]
and that the equation is satisfied in the weak sense of Definition~\ref{def:punctured-weak-solution}.  Non-removability is understood in the
sense of Definition~\ref{def:removable-singularity}.
Assume moreover that the tangential cancellation condition
\eqref{eq:TC} holds.
Then there exists \(R>0\) such that
\[
\Omega=B_R(0)
\]
and \(u\) is radially symmetric and strictly decreasing as a function of
\(|x|\).
\end{theorem}

No connectedness assumption is imposed. This is consistent with the regular
fractional Serrin theorem of Fall--Jarohs \cite{FallJarohs}. In the present
singular setting the moving-plane inclusions obtained from both sides imply,
in every direction, that every affine line section of \(\Omega\) is an
interval. Hence \(\Omega\) is convex, and connectedness follows a posteriori;
the argument is written out at the end of the proof of Theorem~\ref{thm:main}.

The next result removes \eqref{eq:TC} in the ranges in which it
follows from boundary regularity.

\begin{corollary}\label{cor:automatic-TC}
Let the hypotheses of Theorem~\ref{thm:main} hold, except for
\eqref{eq:TC}, and assume in addition that \(\Omega\) is of class
\(C^{2,\alpha}\) for some \(\alpha\in(0,1)\).  If either
\begin{enumerate}[label=\textup{(\roman*)}]
\item \(s>1/2\), or
\item there exists \(a>0\) such that \(f\) is constant on \([0,a]\),
\end{enumerate}
then \eqref{eq:TC} holds automatically.  Consequently,
\(\Omega=B_R(0)\) for some \(R>0\), and \(u\) is radially symmetric and
strictly decreasing as a function of \(|x|\).
\end{corollary}

\begin{proof}
Proposition~\ref{prop:automatic-TC} gives
\eqref{eq:TC}.  The conclusion then follows from Theorem~\ref{thm:main}.
\end{proof}

Once the domain is known to be a ball, the singular part can be described by
the Green function. In the torsion case this gives the following explicit
decomposition.

\begin{corollary}[Torsion classification]\label{cor:torsion-classification}
Let \(0<s<1\) and \(N>2s\), and let \(\Omega\subset\mathbb R^N\) be a bounded
open set of class \(C^{2,\alpha}\), for some \(\alpha\in(0,1)\), such that
\(0\in\Omega\).  Let \(u\) be a positive solution of
\eqref{eq:main-problem}, with \(f\equiv1\), having a non-removable singularity
at the origin, in the regularity class and weak sense specified in Theorem~\ref{thm:main}.  Then, without any additional tangential cancellation
assumption, \(\Omega=B_R(0)\) and
\[
u(x)=\tau_{R}(x)+k G_{R}(x,0),
\]
where \(\tau_{R}\) is the unique solution of the torsion problem
\[
\begin{cases}
(-\Delta)^s\tau_{R}=1 & \text{in }B_R,\\
\tau_{R}=0 & \text{in }\R^N\setminus B_R,
\end{cases}
\]
and \(G_{R}\) is the Green function of \((-\Delta)^s\) in \(B_R\) with zero exterior datum. Moreover,
\[
k=
\frac{sR^{N-s}}{c_{N,s}2^s}
\left(-c-\gamma_{N,s}(2R)^s\right)>0,
\]
with
\[
\gamma_{N,s}
=
\frac{\Gamma(N/2)}{2^{2s}\Gamma(N/2+s)\Gamma(1+s)}, \qquad c_{N,s}=
\frac{\Gamma\left(N/2\right)}
{2^{2s}\pi^{N/2}\Gamma(s)^2}.
\]
Here \(\gamma_{N,s}\) and \(c_{N,s}\) are the positive constants appearing,
respectively, in the torsion function and in the Green representation; see
\cite{Bucur,RosOtonSerra}. In particular,
\[
-c>\gamma_{N,s}(2R)^s.
\]
If the classified family is enlarged to allow a removable singularity, then
equality corresponds to \(k=0\) and \(u=\tau_R\). Finally,
\[
\lim_{x\to0}|x|^{N-2s}u(x)
=k\,c_{N,s}B(s,N/2-s)
=k\,c_{N,s}
\frac{\Gamma(s)\Gamma(N/2-s)}{\Gamma(N/2)}.
\]
\end{corollary}

Section~\ref{sec:preliminaries} fixes the notation and records the required
preliminary results. Section~\ref{sec:boundary-quotient} establishes boundary
regularity in the presence of the interior pole. Section~\ref{sec:comparison-principles} develops the comparison principles for the
reflected singularity. The moving-plane proof of Theorem~\ref{thm:main}
occupies Section~\ref{sec:moving-plane}. Section~\ref{sec:torsion} proves the
torsion classification and its explicit boundary and pole consequences, while
Appendix~\ref{app:strict-TC-example} records the strictness example for
\eqref{eq:TC} stated above.

\section{Preliminaries}
\label{sec:preliminaries}

In this section we collect the functional framework, moving-plane notation, and standard formulas used throughout the paper.

\begin{assumption}
Unless otherwise stated, we work under the hypotheses and notation of Theorem~\ref{thm:main}. In particular,
\[
u\in W^{s,2}_{\loc}(\R^N\setminus\{0\})
\cap L^1_s(\R^N)
\cap C(\R^N\setminus\{0\})
\]
solves \eqref{eq:main-problem} in the weak punctured sense, and the origin is
non-removable. The tangential cancellation hypothesis is invoked only in the
corner step that excludes an orthogonal contact; it is not part of the standing
assumptions in the automatic regularity regimes.
\end{assumption}

\begin{remark}
The three regularity assumptions imposed on \(u\) have different roles. The
local condition
\[
u\in W^{s,2}_{\loc}(\mathbb R^N\setminus\{0\})
\]
is needed in order to use the weak formulation with test functions supported
away from the pole. The tail condition
\[
 u \in L^1_s(\mathbb R^N)
\]
is the natural integrability assumption for the restricted fractional Laplacian,
because it makes the tail terms finite and allows us to
interpret \((-\Delta)^s u\) in a distributional sense. Finally,
\[
u\in C(\mathbb R^N\setminus\{0\})
\]
allows us to compare \(u\) and \(u_{\lambda}\) away from the pole pointwise.
\end{remark}

\begin{remark}
The condition \(N>2s\) has two reasons. It gives the power-type behavior \(|x|^{2s-N}\) of the fundamental singularity and ensures that points have zero fractional \(s\)-capacity. The last property is what permits the capacity cutoffs around the pole and around its reflected image in the moving plane argument.
\end{remark}

\subsection{Fractional Laplacian and weak solutions}

We collect here the normalization of the fractional Laplacian and the
functional setting used throughout the paper.

\begin{definition}
For a sufficiently regular function \(u:\R^N\to\R\), we define
\begin{equation}\label{eq:pointwise-fractional-laplacian}
(-\Delta)^s u(x)
=
k_{N,s}\PV\int_{\R^N}
\frac{u(x)-u(y)}{|x-y|^{N+2s}}\dd y,
\end{equation}
where
\begin{equation}\label{eq:operator-normalization}
k_{N,s}
=
\frac{2^{2s}s\,\Gamma\left(\frac{N+2s}{2}\right)}
{\pi^{N/2}\Gamma(1-s)}
\end{equation}
is the normalization constant; see, for instance, \cite{BucurValdinoci}.
\end{definition}

\begin{definition}
Let \(U\subset\R^N\) be open. The fractional Sobolev space
\(W^{s,2}(U)\) is endowed with the norm
\[
\|v\|_{W^{s,2}(U)}^2
:=\|v\|_{L^2(U)}^2+[v]_{W^{s,2}(U)}^2,
\qquad
[v]_{W^{s,2}(U)}^2
:=\int_U\int_U
\frac{|v(x)-v(y)|^2}{|x-y|^{N+2s}}\dd x\dd y.
\]
We write
\[
H^s(\R^N):=W^{s,2}(\R^N),
\qquad
\|v\|_{H^s(\R^N)}:=\|v\|_{W^{s,2}(\R^N)},
\]
and denote by \(W^{s,2}_{\loc}(U)\) the corresponding local space, endowed
with the usual topology generated by the norms
\(\|\cdot\|_{W^{s,2}(V)}\), \(V\Subset U\).

For an open set \(D\subset\R^N\), we first introduce
\[
\widetilde H^s(D)
:=\overline{C_c^\infty(D)}^{\,H^s(\R^N)},
\qquad
\|v\|_{\widetilde H^s(D)}:=\|v\|_{H^s(\R^N)}.
\]
If \(D\) is bounded with continuous boundary, then
\begin{equation}\label{eq:H0-characterization}
\widetilde H^s(D)
=
\left\{v\in H^s(\R^N):v=0\ \text{a.e. in }\R^N\setminus D\right\};
\end{equation}
see \cite[Theorem 6]{FiscellaServadeiValdinociDensity}. We denote this common
space by \(H_0^s(D)\), with
\[
\|v\|_{H_0^s(D)}:=\|v\|_{H^s(\R^N)}.
\]
Thus, on the Lipschitz domains used below, \(C_c^\infty(D)\) is dense in
\(H_0^s(D)\) for every \(0<s<1\).

We also use the weighted tail space
\[
L_s^1(\R^N)
:=\left\{u:\R^N\to\R:\|u\|_{L_s^1(\R^N)}<\infty\right\},
\qquad
\|u\|_{L_s^1(\R^N)}
:=\int_{\R^N}\frac{|u(x)|}{1+|x|^{N+2s}}\dd x.
\]
\end{definition}

The bilinear form associated with \eqref{eq:pointwise-fractional-laplacian} is
\begin{equation}\label{eq:energy-form}
\mathcal E(u,v)
:=
\frac{k_{N,s}}2
\int_{\R^N}\int_{\R^N}
\frac{(u(x)-u(y))(v(x)-v(y))}{|x-y|^{N+2s}}\dd x\dd y,
\qquad u,v\in H^s(\R^N).
\end{equation}
In particular,
\(\mathcal E(v,v)=\frac{k_{N,s}}2[v]_{W^{s,2}(\R^N)}^2\).
More generally, given a compact set \(\Gamma \subset \R^{N}\), if
\(u\in W^{s,2}_{\loc}(\R^N\setminus\Gamma)\cap L_s^1(\R^N)\) and
\(\phi\in C_c^\infty(\R^N\setminus\Gamma)\), then
\(\mathcal E(u,\phi)\) is well defined and
\begin{equation}\label{eq:energy-distribution-pairing}
\mathcal E(u,\phi)
=
\int_{\R^N}u(-\Delta)^s\phi\dd x.
\end{equation}
For the restricted fractional Laplacian, the Dirichlet condition is understood
in the exterior sense,
\begin{equation}\label{eq:exterior-dirichlet}
u=0\qquad\text{in }\R^N\setminus\Omega,
\end{equation}
a convention that will be essential in the moving-plane comparison.

\begin{definition}\label{def:punctured-weak-solution}
Let \(\Omega \subset \R^N\) be open, let \(\Gamma\subset\Omega\) be compact and let
\(g\in L^1_{\loc}(\Omega\setminus\Gamma)\). A function
\[
u\in W^{s,2}_{\loc}(\R^N\setminus\Gamma)\cap L_s^1(\R^N)
\]
is a weak solution of
\[
(-\Delta)^s u=g\qquad\text{in }\Omega\setminus\Gamma
\]
if
\[
\mathcal E(u,\phi)=\int_\Omega g\phi\dd x
\qquad
\forall\phi\in C_c^\infty(\Omega\setminus\Gamma).
\]
\end{definition}
The tail condition also permits a distributional formulation. For
\(u\in L_s^1(\R^N)\), set
\[
\langle(-\Delta)^su,\phi\rangle
:=\int_{\R^N}u(-\Delta)^s\phi\dd x,
\qquad \phi\in C_c^\infty(\Omega).
\]
The pairing is finite because \(( -\Delta)^s\phi(x)\) decays as
\((1+|x|)^{-N-2s}\). Accordingly, \(( -\Delta)^su=g\) in
\(\mathcal D'(\Omega)\) means that
\[
\langle(-\Delta)^su,\phi\rangle
=\int_\Omega g\phi\dd x
\qquad \forall\phi\in C_c^\infty(\Omega).
\]
By \eqref{eq:energy-distribution-pairing}, the weak and distributional
formulations agree locally away from the singular set.

\begin{definition}\label{def:removable-singularity}
Let \(u\) be a weak solution of
\[
(-\Delta)^s u=f(u)\qquad\text{in }\Omega\setminus\{0\}
\]
in the sense of Definition~\ref{def:punctured-weak-solution}. The singularity
at the origin is called \emph{removable} if there exists
\[
\widetilde u\in W^{s,2}_{\loc}(\R^N)\cap L_s^1(\R^N)\cap C(\R^N)
\]
such that \(\widetilde u=u\) almost everywhere in \(\R^N\setminus\{0\}\) and
\[
\mathcal E(\widetilde u,\phi)
=
\int_\Omega f(\widetilde u)\phi\dd x
\qquad
\forall\phi\in C_c^\infty(\Omega).
\]
Otherwise the isolated singularity is called \emph{non-removable}.
\end{definition}

\subsection{Fractional normal derivative and boundary quotient}

Let \(\Omega\subset\R^N\) be of class \(C^2\), set
\(\delta(x):=\dist(x,\R^N\setminus\Omega)\), and denote by \(\eta(x_0)\)
the outer unit normal at \(x_0\in\partial\Omega\). Whenever the limit exists,
we define the fractional normal derivative by
\begin{equation}\label{eq:fractional-normal-derivative}
(\partial_\eta)^s u(x_0)
:=
-\lim_{t\to0^+}\frac{u(x_0-t\eta(x_0))}{t^s}.
\end{equation}
If the quotient \(u/\delta^s\) extends continuously to the boundary, then the
tubular-neighborhood representation of the distance function gives
\[
(\partial_\eta)^s u(x_0)
=-\left.\frac{u}{\delta^s}\right|_{\partial\Omega}(x_0).
\]
Thus the overdetermined condition in \eqref{eq:main-problem} is equivalent to
prescribing the constant boundary trace \(u/\delta^s=-c\). The required
continuity of this quotient in the presence of the interior pole is established
in Section~\ref{sec:boundary-quotient} by reducing, through a cutoff argument,
to the boundary regularity theorem of Ros-Oton--Serra
\cite[Theorem 1.2]{RosOtonSerra}.

\subsection{Capacity and moving-plane framework}

\paragraph{Capacity cutoffs.}
Since \(N>2s\), points have zero fractional \(H^s\)-capacity. More
precisely, for a compact set \(E\subset\mathbb R^N\) we use the notion of capacity
\[
\operatorname{Cap}_s(E)
:=
\inf\left\{
\|\varphi\|_{H^s(\mathbb R^N)}^2:
\varphi\in C_c^\infty(\mathbb R^N),
\varphi\geq1\ \text{in a neighborhood of }E
\right\}.
\]
We record the explicit cutoffs that will be used below.  Fix
\(\eta\in C_c^\infty(B_2)\) such that
\[
0\leq\eta\leq1,
\qquad
\eta\equiv1\quad\text{in }B_1,
\]
and, for \(a\in\mathbb R^N\), define
\[
\eta_{a,\varepsilon}(x)
:=
\eta\!\left(\frac{x-a}{\varepsilon}\right),
\qquad
\psi_{a,\varepsilon}:=1-\eta_{a,\varepsilon}.
\]
Then
\begin{equation}\label{eq:capacity-cutoffs}
\begin{gathered}
0\leq\psi_{a,\varepsilon}\leq1,
\qquad
\psi_{a,\varepsilon}=0\quad\text{in }B_\varepsilon(a),
\qquad
\psi_{a,\varepsilon}\longrightarrow1
\quad\text{a.e. in }\mathbb R^N,
\\
\|\eta_{a,\varepsilon}\|_{L^2(\mathbb R^N)}^2
=
\varepsilon^N\|\eta\|_{L^2(\mathbb R^N)}^2,
\qquad
[\eta_{a,\varepsilon}]_{H^s(\mathbb R^N)}^2
=
\varepsilon^{N-2s}[\eta]_{H^s(\mathbb R^N)}^2
\longrightarrow0.
\end{gathered}
\end{equation}
In particular,
\(\|\eta_{a,\varepsilon}\|_{H^s(\mathbb R^N)}\to0\), which proves
\(\operatorname{Cap}_s(\{a\})=0\).  Notice that
\(\psi_{a,\varepsilon}\) is not compactly supported.  This is harmless in our
applications because it is always multiplied by a function supported in a
bounded domain.  Moreover,
\[
\psi_{a,\varepsilon}(x)-\psi_{a,\varepsilon}(y)
=-
\bigl(\eta_{a,\varepsilon}(x)-\eta_{a,\varepsilon}(y)\bigr),
\]
so every commutator error involving differences of
\(\psi_{a,\varepsilon}\) is controlled by the seminorm in
\eqref{eq:capacity-cutoffs}.

\paragraph{Reflection notation.}
We next fix the notation for the moving-plane procedure. For \(\lambda<0\), the original pole lies outside the moving half-space, whereas the reflected pole \(0^\lambda\) may lie inside it. For \(\lambda\in\R\), let
\[
T_\lambda:=\{x\in\R^N:x_1=\lambda\},
\quad
\Sigma_\lambda:=\{x\in\R^N:x_1<\lambda\},
\quad
\Omega_\lambda:=\Omega\cap\Sigma_\lambda,
\quad \Omega'_{\lambda} := R_{\lambda}(\Omega_{\lambda}).
\]
The reflection of \(x=(x_1,x_2,\dots,x_N)\) with respect to \(T_\lambda\) is
\[
x^\lambda := R_{\lambda}(x) =(2\lambda-x_1,x_2,\dots,x_N).
\]
We set
\[
u_\lambda(x):=u(x^\lambda),
\qquad
w_\lambda:=u_\lambda-u.
\]
The singular point \(0\) is reflected into
\[
0^\lambda=(2\lambda,0,\dots,0).
\]
Thus \(w_\lambda\) is naturally defined on
\[
\mathbb R^N\setminus\{0,0^\lambda\}.
\]
For \(\lambda<0\), the original pole lies outside \(\Sigma_\lambda\), so
pointwise statements in the moving half-space are understood on
\(\Sigma_\lambda\setminus\{0^\lambda\}\).

\paragraph{Antisymmetric supersolutions.}
For later use, we recall the local energy class employed by Fall--Jarohs.  If \(D\subset\mathbb R^N\) is bounded and open, we write
\[
\mathcal D^s(D)
:=
\left\{
v:\mathbb R^N\to\mathbb R\ \text{measurable}:
\mathcal E(v,\phi)\ \text{is finite for every }\phi\in H_0^s(D)
\right\}.
\]
The class \(\mathcal D^s(D)\) is used only to formulate weak supersolution
properties in the sense of Fall--Jarohs.
When $v \notin H^{s}(\R^{N})$, the pairing is interpreted by localizing around $\operatorname{supp} \phi$: the local part is controlled by the $H^{s}$-energy of $v$, while the interaction with the complement is controlled by its tail.

The comparison arguments below use the weak and strong maximum principles for
antisymmetric supersolutions, together with the fractional Hopf and corner
lemmas of Fall--Jarohs \cite{FallJarohs}. The small-domain principle is
triggered by the spectral condition
\begin{equation}\label{eq:FJ-spectral-condition}
\|c^+\|_{L^\infty(D)}<\lambda_1(D),
\end{equation}
with
\[
\lambda_1(D) :=
\inf_{v\in H^s_0(D)\setminus\{0\}}
\frac{\mathcal E(v,v)}{\int_D v^2\dd x},
\]
where \(\lambda_{1}(D)\) is the first Dirichlet eigenvalue of \((-\Delta)^{s}\) on \(D\).
We will also use the following estimate, see \cite[Section 2]{FallJarohs}:
\begin{equation}\label{eq:eigenvalue-measure-bound}
\lambda_1(D)\ge \alpha_{N,s}|D|^{-2s/N},
\end{equation}
where \(\alpha_{N,s} = k_{N,s} \frac{N}{2s}
|B_{1}(0)|^{1+\frac{2s}{N}}\), and \(k_{N,s}\) is given in
\eqref{eq:operator-normalization}. Thus \(\lambda_1(D)\to+\infty\) as
\(|D|\to0\), which is precisely what is needed to enforce
\eqref{eq:FJ-spectral-condition} on sufficiently small sets.

\subsection{Explicit formulas in balls}

We record the explicit formulas in balls used later. For \(B_R=B_R(0)\), the fractional torsion function is
\[
\tau_{R}(x)=\gamma_{N,s}(R^2-|x|^2)^s_+,
\qquad
\gamma_{N,s}
=
\frac{\Gamma(N/2)}{2^{2s}\Gamma(N/2+s)\Gamma(1+s)}.
\]
For \(N>2s\), the Green function of the restricted fractional Laplacian in \(B_R\) is given by the Boggio formula; see, for instance, \cite{AbatangeloJarohsSaldana, Bucur}:
\[
G_{R}(x,y)=
c_{N,s}|x-y|^{2s-N}
\int_0^{\rho_R(x,y)}
\frac{t^{s-1}}{(1+t)^{N/2}}\dd t, \qquad c_{N,s}=
\frac{\Gamma\left(N/2\right)}
{2^{2s}\pi^{N/2}\Gamma(s)^2}
\]
where
\[
\rho_R(x,y)=
\frac{(R^2-|x|^2)(R^2-|y|^2)}
{R^2|x-y|^2}.
\]
With the normalization \eqref{eq:operator-normalization}, this Green function
satisfies
\[
(-\Delta)^sG_R(\cdot,y)=\delta_y
\quad\text{in }\mathcal D'(B_R),
\qquad
G_R(\cdot,y)=0
\quad\text{in }\mathbb R^N\setminus B_R.
\]
In particular,
\[
G_{R}(x,0)
=
c_{N,s}|x|^{2s-N}
\int_0^{(R^2-|x|^2)/|x|^2}
\frac{t^{s-1}}{(1+t)^{N/2}}\dd t.
\]
We stress that the constants \(k_{N,s}\) and \(c_{N,s}\) have different roles: \(k_{N,s}\) normalizes the operator and the energy form, whereas \(c_{N,s}\) is the coefficient in the Boggio representation of the Green function.

\section{Regularity}
\label{sec:boundary-quotient}

The aim of this section is to justify rigorously the boundary quotient used in
\eqref{eq:fractional-normal-derivative}. The difficulty is that the solution
appearing in \eqref{eq:main-problem} need not be bounded near the interior
pole, whereas the boundary regularity theorem of Ros-Oton--Serra \cite[Theorem 1.2]{RosOtonSerra} applies to energy solutions with bounded right-hand side. We therefore cut off the pole,
prove that the truncated function solves a global Dirichlet problem with
bounded right-hand side, and then transfer the regularity back to the
original solution in the boundary neighborhood where the cutoff is identically one.

We first establish local regularity on compact sets separated from both the
boundary and the pole.

\begin{lemma}
\label{lem:interior-holder}
Let \(0<s<1\), \(N>2s\), \(R>0\), and
\(B_{4R}(x_0)\Subset\Omega\setminus\{0\}\). Assume that
\(u\in L^1_s(\mathbb R^N)\) and that \((-\Delta)^s u=g\) in
\(B_{4R}(x_0)\) in the sense of distributions, with
\(g\in L^\infty(B_{4R}(x_0))\). Then
\(u\in C^\beta(B_R(x_0))\) for every
\(0<\beta<\min\{2s,1\}\). In particular, one may choose
\(\beta>2s-1\).
\end{lemma}

\begin{proof}
Choose \(\eta\in C_c^\infty(B_{3R}(x_0))\), with
\(\eta\equiv1\) in \(B_{2R}(x_0)\), and extend \(G:=\eta g\) by zero. Let
\[
\Phi_s(x)=a_{N,s}|x|^{2s-N},
\qquad
(-\Delta)^s\Phi_s=\delta_0
\quad\text{in }\mathcal D'(\mathbb R^N),
\]
where \(a_{N,s}>0\); see
\cite[Theorem 2.3 and Proposition 2.4]{Bucur}, together with
\cite[Theorem 1.1]{Kwasnicki} for the equivalence between the
singular-integral, Fourier, and distributional definitions of \((-\Delta)^{s}\). Set
\[
P:=\Phi_s*G.
\]
The kernel is locally integrable, so \(P\) is well defined and
\[
(-\Delta)^sP=G
\quad\text{in }\mathcal D'(\mathbb R^N).
\]
Standard estimates for the Riesz potential give, for
\(x,x'\in B_{2R}(x_0)\) and \(d=|x-x'|\in(0,1)\),
\[
|P(x)-P(x')|
\leq C\|G\|_\infty
\begin{cases}
d^{2s},&2s<1,\\
d(1+|\log d|),&2s=1,\\
d,&2s>1.
\end{cases}
\]
Thus \(P\in C^\beta(B_{2R}(x_0))\) for every
\(0<\beta<\min\{2s,1\}\).

Now set \(H:=u-P\). Since \(G=g\) in \(B_{2R}(x_0)\),
\[
(-\Delta)^sH=0
\quad\text{in }B_{2R}(x_0)
\]
in the distributional sense. Moreover, \(H\in L^1_s(\mathbb R^N)\). Indeed,
\(u\in L^1_s(\mathbb R^N)\), while
\[
|P(x)|\leq C|x|^{2s-N}
\qquad\text{for large }|x|
\]
gives \(P\in L^1_s(\mathbb R^N)\). After translation and rescaling, the local
regularity theorem for very weak \(s\)-harmonic functions
\cite[Main Theorem, item (2)]{CarbottiCitoLaMannaPallara} shows that \(H\) is
real-analytic in \(B_{2R}(x_0)\). Hence
\(u=P+H\in C^\beta(B_R(x_0))\) for every
\(0<\beta<\min\{2s,1\}\), and such a \(\beta\) can be chosen larger than
\(2s-1\).
\end{proof}

We next replace the unbounded singular solution by a bounded truncated
function that agrees with it near the boundary. Lemma~\ref{lem:interior-holder} controls the cutoff commutator in the compact
intermediate region.

\begin{proposition}
\label{prop:boundary-quotient}
Let \(\Omega\subset\mathbb R^N\) be a bounded open set of class \(C^2\), with
\(0\in\Omega\), and set \(d_0:=\operatorname{dist}(0,\partial\Omega)>0\).
Let \(F\in L^\infty_{\loc}(\overline\Omega\setminus\{0\})\), meaning that
\(F\in L^\infty(K\cap\Omega)\) for every compact set
\(K\subset\overline\Omega\setminus\{0\}\). Assume that
\(u\in L^1_s(\mathbb R^N)\cap
W^{s,2}_{\mathrm{loc}}(\Omega\setminus\{0\})\cap
C(\mathbb R^N\setminus\{0\})\) satisfies, in the weak sense,
\[
(-\Delta)^s u=F \quad\text{in }\Omega\setminus\{0\},
\qquad
u=0 \quad\text{in }\mathbb R^N\setminus\Omega.
\]
Set \(\delta(x):=\operatorname{dist}(x,\mathbb R^N\setminus\Omega)\). Then
there exist \(\rho>0\), \(\alpha\in(0,1)\), and
\(\psi\in C^\alpha(\overline{U_\rho})\), where
\(U_\rho:=\{x\in\Omega:\delta(x)<\rho\}\), such that
\(u=\delta^s\psi\) in \(U_\rho\). In particular, \(u/\delta^s\) is
H\"older continuous in a neighborhood of \(\partial\Omega\).
\end{proposition}

\begin{proof}
We remove the pole by a cutoff, but we keep track of the equation satisfied by
that cutoff function on the whole domain.  Choose \(r>0\) such that
\(4r<d_0\), and let \(\chi\in C^\infty(\mathbb R^N)\) satisfy
\[
0\leq\chi\leq1,
\qquad
\chi=0\ \text{in }B_r,
\qquad
\chi=1\ \text{in }\mathbb R^N\setminus B_{2r}.
\]
Set
\[
w:=\chi u.
\]
Since the cutoff removes the only singular point, while \(u\) is continuous on
\(\mathbb R^N\setminus\{0\}\) and vanishes outside the bounded set \(\Omega\),
we have
\[
w\in L^\infty(\mathbb R^N)\cap C(\mathbb R^N)\cap L_s^1(\mathbb R^N),
\qquad
w=0\ \text{in }\mathbb R^N\setminus\Omega.
\]
We claim that
\begin{equation}\label{eq:truncated-global-equation}
(-\Delta)^s w=G\quad\text{in }\Omega,
\qquad G\in L^\infty(\Omega),
\end{equation}
in the weak sense.  We prove the claim by identifying the equation
of \(w\) separately near the boundary, in an intermediate region, and near the
pole, and then patching the three identities.

Choose \(\rho>0\) so small that
\[
\overline{U_\rho}\cap\overline{B_{2r}}=\emptyset,
\qquad
U_\rho:=\{x\in\Omega:\delta(x)<\rho\}.
\]
In particular, \(\chi\equiv1\) and hence \(w=u\) in \(U_\rho\).

\smallskip
\noindent\emph{Step 1: the boundary collar.}
Let \(\phi\in C_c^\infty(U_\rho)\) and set \(S:=\supp\phi\).  Since
\(S\Subset\Omega\setminus\{0\}\), the function \(\phi=\chi\phi\) is an
admissible test function in the equation for \(u\), and therefore
\[
\mathcal E(u,\chi\phi)=\mathcal E(u,\phi)
=\int_\Omega F\phi\,\dd x.
\]
We now compute explicitly the difference between the pairings with \(w=\chi u\)
and with \(u\).  For \(\varepsilon>0\), let
\[
\mathcal E_\varepsilon(a,b)
:=\frac{k_{N,s}}2
\iint_{|x-y|>\varepsilon}
\frac{(a(x)-a(y))(b(x)-b(y))}{|x-y|^{N+2s}}\,\dd x\,\dd y.
\]
Thus \(\mathcal E_\varepsilon\) is the energy pairing with the kernel truncated
at distance \(\varepsilon\).  Expanding the two numerators gives
\begin{align*}
&\bigl(\chi(x)u(x)-\chi(y)u(y)\bigr)
   \bigl(\phi(x)-\phi(y)\bigr)\\
&\quad-
\bigl(u(x)-u(y)\bigr)
   \bigl(\chi(x)\phi(x)-\chi(y)\phi(y)\bigr)\\
&=(\chi(x)-\chi(y))
  \bigl(u(y)\phi(x)-u(x)\phi(y)\bigr).
\end{align*}
Hence, by symmetry of the kernel and by exchanging \(x\) and \(y\) in the
second term,
\begin{align*}
&\mathcal E_\varepsilon(w,\phi)
 -\mathcal E_\varepsilon(u,\chi\phi)\\
&\qquad =k_{N,s}\int_{\mathbb R^N}\phi(x)
 \int_{|x-y|>\varepsilon}
 \frac{(\chi(x)-\chi(y))u(y)}{|x-y|^{N+2s}}\,\dd y\,\dd x.
\end{align*}
If \(\phi(x)\neq0\), then \(x\in S\subset U_\rho\) and
\(\chi(x)=1\).  Moreover, \(1-\chi\) is supported in \(B_{2r}\), while
\(S\) has positive distance from \(B_{2r}\).  Thus the last integral has no
singularity at \(y=x\) and, for \(\varepsilon\) sufficiently small, the condition $|x-y|> \varepsilon$ is automatically satisfied on the support of the integrand.
Passing to the limit gives
\[
\mathcal E(w,\phi)-\mathcal E(u,\chi\phi)
=
 k_{N,s}\int_{U_\rho}\phi(x)
 \int_{\mathbb R^N}
 \frac{(1-\chi(y))u(y)}{|x-y|^{N+2s}}\,\dd y\,\dd x.
\]
Consequently,
\begin{equation}\label{eq:truncated-collar-equation}
(-\Delta)^s w=F+h\quad\text{in }U_\rho,
\qquad
h(x):=k_{N,s}\int_{\mathbb R^N}
\frac{(1-\chi(y))u(y)}{|x-y|^{N+2s}}\,\dd y.
\end{equation}
This is the precise nonlocal correction produced by cutting off the pole.  To
see that it is bounded, set
\[
d_1:=\dist\bigl(\overline{U_\rho},\overline{B_{2r}}\bigr)>0.
\]
Since \(\supp(1-\chi)\subset B_{2r}\), for every \(x\in U_\rho\),
\[
|h(x)|
\leq k_{N,s}d_1^{-N-2s}
\int_{B_{2r}}|u(y)|\,\dd y.
\]
The last integral is finite because \(u\in L_s^1(\mathbb R^N)\). Thus
\(h\in L^\infty(U_\rho)\).  Moreover,
\(\overline{U_\rho}\subset\overline\Omega\setminus\{0\}\), so the
assumption on \(F\) gives \(F\in L^\infty(U_\rho)\).  Hence
\(F+h\in L^\infty(U_\rho)\).

\smallskip
\noindent\emph{Step 2: the intermediate region.}
Set
\[
D_m:=\{x\in\Omega:|x|>r/2,\ \delta(x)>\rho/2\}.
\]
Then \(\overline{D_m}\Subset\Omega\setminus\{0\}\).  Choose an open set
\(D'\) such that
\[
\overline{D_m}\cup\supp\nabla\chi\Subset D'
\Subset\Omega\setminus\{0\}.
\]
On balls compactly contained in \(D'\), the weak equation for \(u\) implies
the corresponding distributional identity and the right-hand side is bounded.
By Lemma~\ref{lem:interior-holder}, after a finite covering of
\(\overline{D'}\), there exists
\[
\beta>\max\{0,2s-1\}
\]
such that \(u\in C^\beta(\overline{D'})\).

Let \(\phi\in C_c^\infty(D_m)\).  Repeating the truncated-kernel computation
above, but without using \(\chi \equiv 1\) on \(\supp\phi\), yields
\begin{equation}\label{eq:differenza}
\mathcal E_\varepsilon(w,\phi)-\mathcal E_\varepsilon(u,\chi\phi)
=k_{N,s}\int_{\mathbb R^N}\phi(x)
\int_{|x-y|>\varepsilon}
\frac{(\chi(x)-\chi(y))u(y)}{|x-y|^{N+2s}}\,\dd y\,\dd x.
\end{equation}
For \(x\in\supp\phi\), write
\[
u(y)=u(x)-(u(x)-u(y)).
\]
Then the inner integral of \eqref{eq:differenza} becomes
\begin{equation}\label{eq:commutatore}
\begin{split}
&u(x)k_{N,s}\PV\int_{\mathbb R^N}
\frac{\chi(x)-\chi(y)}{|x-y|^{N+2s}}\,\dd y\\
&\quad-k_{N,s}\PV\int_{\mathbb R^N}
\frac{(\chi(x)-\chi(y))(u(x)-u(y))}
{|x-y|^{N+2s}}\,\dd y\\
&=u(x)(-\Delta)^s\chi(x)-\mathcal C_\chi[u](x),
    \end{split}
\end{equation}
where
\[
\mathcal C_\chi[u](x)
:=k_{N,s}\PV\int_{\mathbb R^N}
\frac{(\chi(x)-\chi(y))(u(x)-u(y))}
{|x-y|^{N+2s}}\,\dd y.
\]
Using \eqref{eq:commutatore} and passing to the limit
\(\varepsilon\to0^+\) in \eqref{eq:differenza} therefore gives the product identity
\begin{equation}\label{eq:cutoff-product-identity}
(-\Delta)^s(\chi u)
=
\chi F+u(-\Delta)^s\chi-\mathcal C_\chi[u]
\qquad\text{in }D_m
\end{equation}
in the weak sense.

We verify that the commutator is uniformly bounded in \(D_m\), which is the
only non-immediate term in \eqref{eq:cutoff-product-identity}.  Choose
\(d_*>0\) such that \(B_{d_*}(x)\subset D'\) for every
\(x\in\overline{D_m}\), and split
\[
\mathcal C_\chi[u](x)
=\mathcal C_{\chi,\mathrm{loc}}[u](x)
 +\mathcal C_{\chi,\mathrm{tail}}[u](x)
\]
according to \(B_{d_*}(x)\) and its complement.  For the local part,
\(\chi\in C^1\) and \(u\in C^\beta(D')\) give
\[
|\chi(x)-\chi(y)|\leq C|x-y|,
\qquad
|u(x)-u(y)|\leq C|x-y|^\beta,
\]
and hence, uniformly for \(x\in D_m\),
\[
|\mathcal C_{\chi,\mathrm{loc}}[u](x)|
\leq C\int_0^{d_*}t^{\beta-2s}\,\dd t<\infty,
\]
precisely because \(\beta>2s-1\) by Lemma \ref{lem:interior-holder}.

For the tail part, using boundedness of \(\chi\) and of \(u\) on
\(\overline{D_m}\),
\begin{align*}
|\mathcal C_{\chi,\mathrm{tail}}[u](x)|
&\leq C|u(x)|
\int_{\mathbb R^N\setminus B_{d_*}(x)}
\frac{\dd y}{|x-y|^{N+2s}}\\
&\quad+C\int_{\mathbb R^N\setminus B_{d_*}(x)}
\frac{|u(y)|}{|x-y|^{N+2s}}\,\dd y.
\end{align*}
The first integral is bounded uniformly in \(x\).  For the second one, since
\(D_m\) is bounded, there exists \(C_*>0\), depending only on \(D_m\) and
\(d_*\), such that
\[
\frac1{|x-y|^{N+2s}}
\leq \frac{C_*}{1+|y|^{N+2s}}
\qquad
\text{for }x\in D_m,\quad |x-y|\geq d_*.
\]
Therefore the tail is controlled by
\(\|u\|_{L_s^1(\mathbb R^N)}\), and
\[
\mathcal C_\chi[u]\in L^\infty(D_m).
\]
Since \(F\) and \(u\) are bounded on \(D_m\), and
\(( -\Delta)^s\chi \in L^\infty(\mathbb R^N)\),
we conclude that
\[
G_m:=\chi F+u(-\Delta)^s\chi-\mathcal C_\chi[u]
\in L^\infty(D_m).
\]

\smallskip
\noindent\emph{Step 3: a neighborhood of the pole.}
Let
\[
D_0:=B_{3r/4}(0)\Subset\Omega.
\]
Since \(w=0\) in \(B_r\), if \(x\in D_0\) and \(w(y)\neq0\), then
\(|x-y|\geq r/4\).  Hence no principal value is needed and, for
\(\phi\in C_c^\infty(D_0)\), symmetry of the kernel gives
\begin{align*}
\mathcal E(w,\phi)
&=-k_{N,s}\int_{D_0}\phi(x)
\int_{\mathbb R^N}
\frac{w(y)}{|x-y|^{N+2s}}\,\dd y\,\dd x\\
&=\int_{D_0}G_0(x)\phi(x)\,\dd x,
\end{align*}
where
\[
G_0(x):=-k_{N,s}\int_{\mathbb R^N}
\frac{w(y)}{|x-y|^{N+2s}}\,\dd y.
\]
The separation \(|x-y|\geq r/4\), together with the boundedness and compact
support of \(w\), yields \(G_0\in L^\infty(D_0)\).

\smallskip
\noindent\emph{Step 4: patching the local equations.}
The three open sets just introduced cover \(\Omega\): indeed,
\[
\Omega=D_0\cup D_m\cup U_\rho.
\]
Let \(\zeta_0,\zeta_m,\zeta_\rho\in C^\infty(\Omega)\)
be a partition of unity subordinate to this cover.  For an arbitrary
\(\phi\in C_c^\infty(\Omega)\), set
\[
\phi_0:=\zeta_0\phi,
\qquad
\phi_m:=\zeta_m\phi,
\qquad
\phi_\rho:=\zeta_\rho\phi.
\]
Then \(\phi=\phi_0+\phi_m+\phi_\rho\), with the three terms supported in
\(D_0,D_m,U_\rho\), respectively.  By linearity of the whole-space pairing and
the three local identities,
\begin{align*}
\mathcal E(w,\phi)
&=\int_\Omega G_0\phi_0\,\dd x
 +\int_\Omega G_m\phi_m\,\dd x
 +\int_\Omega(F+h)\phi_\rho\,\dd x\\
&=\int_\Omega G\phi\,\dd x,
\end{align*}
where, after extending the local right-hand sides by zero outside their
respective sets,
\[
G:=\zeta_0G_0+\zeta_mG_m+\zeta_\rho(F+h).
\]
Each local right-hand side is bounded on the corresponding member of the
cover; therefore \(G\in L^\infty(\Omega)\).  This proves
\eqref{eq:truncated-global-equation}.

It remains to place \(w\) in the variational class required by the boundary regularity theorem of
Ros-Oton--Serra \cite[Theorem 1.2]{RosOtonSerra}. Let
\(v\in H_0^s(\Omega)\) be the unique Lax--Milgram solution of
\[
(-\Delta)^s v=G\quad\text{in }\Omega,
\qquad
v=0\quad\text{in }\mathbb R^N\setminus\Omega.
\]
Since a bounded \(C^2\) open set has finitely many connected components with
pairwise separated closures, write \(v_j:=\mathbf1_{\Omega_j}v\). By the same
componentwise identity displayed below in \eqref{eq:component-equation}, with
\(w\) replaced by \(v\), the right-hand side for \(v_j\) is \(G|_{\Omega_j}\)
plus the interaction with the other components, which is smooth on
\(\overline{\Omega_j}\). Thus \cite[Theorem 1.2]{RosOtonSerra} applies to every
\(v_j\).
Since \(\delta\) agrees in \(\Omega_j\) with the distance to
\(\mathbb R^N\setminus\Omega_j\), and \(\Omega\) has finitely many connected
components with pairwise separated closures, we obtain
\[
\frac{v}{\delta^s}\in C^\alpha(\overline\Omega)
\]
for some \(\alpha\in(0,1)\).

Set \(z:=w-v\).  Then \(z\in L_s^1(\mathbb R^N)\cap C(\mathbb R^N)\),
\(z=0\) in \(\mathbb R^N\setminus\Omega\), and
\(( -\Delta)^s z=0\) in \(\mathcal D'(\Omega)\). The interior regularity theorem for
very weak \(s\)-harmonic functions
\cite[Main Theorem, item (2)]{CarbottiCitoLaMannaPallara}
gives
\[
z\in C^\infty_{\loc}(\Omega).
\]
In particular, the equation holds pointwise in \(\Omega\).  If
\(z\) had a positive maximum at an interior point \(x_0\), then the pointwise
formula would give
\[
0=(-\Delta)^s z(x_0)
=k_{N,s}\PV\int_{\mathbb R^N}
\frac{z(x_0)-z(y)}{|x_0-y|^{N+2s}}\,\dd y>0,
\]
since \(z=0\) outside the bounded domain and the maximum is positive.  The
same argument applied to \(-z\) shows that \(z\equiv0\).  Hence
\(w=v\in H_0^s(\Omega)\), and therefore
\[
\frac{w}{\delta^s}\in C^\alpha(\overline\Omega).
\]
Finally, \(w=u\) in \(U_\rho\), so the conclusion follows with
\(\psi=u/\delta^s\) in that collar.
\end{proof}

Proposition~\ref{prop:boundary-quotient} yields only H\"older continuity of
the boundary quotient and therefore does not, by itself, imply
\eqref{eq:TC}.  We first record the elementary relation between the
regularized distance used in the higher-order boundary estimates and the
usual distance to the boundary.

\paragraph{Regularized versus actual distance.}
We shall use the following elementary fact.  If \(\Omega\) is of class
\(C^{2,\alpha}\), \(d\) is a regularized distance in the sense of
\cite{AbatangeloRosOton}, and \(0<\varepsilon<\alpha\), then, in a
sufficiently thin boundary collar,
\begin{equation}\label{eq:regularized-distance-ratio}
\frac d\delta\in C^{1,\varepsilon},
\qquad
0<c\leq\frac d\delta\leq C.
\end{equation}
Indeed, both \(d\) and \(\delta\) are \(C^{2,\varepsilon}\) near the
boundary, vanish on \(\partial\Omega\), and are comparable.  Thus the common first-order vanishing can be factored
out in a normal coordinate.  Fix \(Q\in\partial\Omega\) and, after a rotation
of coordinates, assume that \(\partial_{x_N}\delta\neq0\) in a neighborhood
of \(Q\).  By the inverse function theorem, \((z',t)=(x',\delta(x))\)
defines a \(C^{2,\varepsilon}\) local system of coordinates.  In these
coordinates the boundary is represented by \(\{t=0\}\), while the last
coordinate is exactly the actual distance, \(t=\delta(x)\).

Writing \(x=x(z',t)\) for the inverse coordinate map and setting
\[
D(z',t):=d(x(z',t)),
\]
we have \(D\in C^{2,\varepsilon}\) and \(D(z',0)=0\).  Hence the fundamental
theorem of calculus in the normal variable gives
\[
D(z',t)
=
t\int_0^1\partial_tD(z',\theta t)\,\dd\theta,
\qquad
\frac{D(z',t)}t
=
\int_0^1\partial_tD(z',\theta t)\,\dd\theta.
\]
The right-hand side extends with \(C^{1,\varepsilon}\)-regularity to
\(t=0\), since \(D\in C^{2,\varepsilon}\).  Since \(t=\delta(x)\), this
extension is precisely the quotient \(d/\delta\). A finite covering of \(\partial\Omega\) proves
\eqref{eq:regularized-distance-ratio} in a uniform boundary collar, and the
positive upper and lower bounds follow from the comparability of \(d\) and
\(\delta\).  In particular, \((d/\delta)^s\in C^{1,\varepsilon}\), so
boundary regularity for quotients by \(d^s\) transfers directly to quotients
by \(\delta^s\).

We can now identify the regimes in which the tangential cancellation follows
from higher boundary regularity. To reach the boundary regularity needed for
\eqref{eq:TC}, we first bootstrap the solution on compact sets away from the
pole.

\begin{lemma}
\label{lem:interior-bootstrap}
Assume the hypotheses of Theorem~\ref{thm:main}, except for
\eqref{eq:TC}. Then, for every compact set
\(K_0\Subset\Omega\setminus\{0\}\),
\begin{equation}\label{eq:interior-bootstrap}
u\in C^\theta(K_0)
\qquad\text{for every noninteger }\theta<1+2s.
\end{equation}
\end{lemma}

\begin{proof}
Fix first \(0<\gamma_0<\min\{2s,1\}\). By
Lemma~\ref{lem:interior-holder}, after slightly reducing \(\gamma_0\) if
necessary, \(u\in C^{\gamma_0}\) on every compact subset of
\(\Omega\setminus\{0\}\).

We first record one bootstrap step. Let
\[
V'\Subset W\Subset V\Subset\Omega\setminus\{0\}
\]
and suppose that \(u\in C^\gamma(\overline V)\) for some
\(0<\gamma<1\). Choose \(\eta\in C_c^\infty(V)\) such that
\(\eta\equiv1\) in an open neighborhood of \(\overline W\), and set
\(v:=\eta u\), with \(v=0\) at the pole. Since
\(\supp\eta\Subset V\) and \(u\in C^\gamma(\overline V)\), the product
extends by zero to a compactly supported function satisfying
\[
v\in C^\gamma(\mathbb R^N).
\]

We next identify the equation satisfied by \(v\) in \(W\). The same
technique used to prove \eqref{eq:truncated-collar-equation} applies here. Indeed, for
\(x\in W\) we have \(\eta(x)=1\), and hence
\[
u(x)-\eta(y)u(y)
=
\bigl(u(x)-u(y)\bigr)
+
(1-\eta(y))u(y).
\]
Therefore, in the weak sense in \(W\),
\begin{equation}\label{eq:localized-bootstrap-equation}
(-\Delta)^s v
=
f(u)+H_\eta,
\end{equation}
where
\[
H_\eta(x)
:=
k_{N,s}\int_{\mathbb R^N}
\frac{(1-\eta(y))u(y)}{|x-y|^{N+2s}}\,\dd y.
\]
No principal value is needed in the definition of \(H_\eta\). Since
\(\eta\equiv1\) in a neighborhood of \(\overline W\),
\[
d:=\dist\bigl(\overline W,\supp(1-\eta)\bigr)>0.
\]
Thus the kernel is nonsingular on
\(\overline W\times\supp(1-\eta)\). More precisely, for every multi-index
\(\alpha\),
\[
\left|D_x^\alpha |x-y|^{-N-2s}\right|
\leq C_\alpha |x-y|^{-N-2s-|\alpha|}.
\]
On bounded sets the right-hand side is uniformly bounded for
\(x\in\overline W\) and \(y\in\supp(1-\eta)\), while for \(|y|\) large,
uniformly in \(x\in\overline W\),
\[
|x-y|^{-N-2s-|\alpha|}
\leq \frac{C_\alpha}{1+|y|^{N+2s}}.
\]
Since \(u\in L_s^1(\mathbb R^N)\), dominated convergence therefore permits
differentiation under the integral sign to arbitrary order and gives
\[
H_\eta\in C^\infty(\overline W).
\]
Since \(f\) is Lipschitz on the compact range of
\(u|_{\overline V}\), it follows that
\[
f(u)+H_\eta\in C^\gamma(\overline W).
\]

We may now use the interior Schauder estimate for the fractional Laplacian
from \cite[Proposition 2.2]{RosOtonSerra}. Since
\(v\in C^\gamma(\mathbb R^N)\) and
\(v=u\) in \(W\), a finite covering of \(\overline{V'}\) yields
\begin{equation}\label{eq:one-schauder-step}
u\in C^{\gamma+2s}(\overline{V'})
\end{equation}
whenever \(\gamma+2s\notin\mathbb N\). If an integer exponent is met, we
first replace \(\gamma\) by any slightly smaller positive exponent.

Choose now a finite chain of nested open sets between \(K_0\) and a compact
subset of \(\Omega\setminus\{0\}\) on which the initial
\(C^{\gamma_0}\)-estimate holds. Repeatedly apply
\eqref{eq:one-schauder-step}, lowering the H\"older exponent by an
arbitrarily small amount whenever necessary to avoid integers. Since each
step gains essentially \(2s>0\), after finitely many steps one obtains
\[
u\in C^{\theta_*}(K_1)
\]
for some \(K_0\Subset K_1\Subset\Omega\setminus\{0\}\) and some
noninteger \(\theta_*>1\). In particular, \(u\) is Lipschitz on \(K_1\).
Therefore \(f(u)\) is Lipschitz there and hence belongs to
\(C^\gamma(K_1)\) for every \(0<\gamma<1\). Given any noninteger
\(\theta<1+2s\), choose \(\gamma\in(0,1)\) such that
\[
\theta<\gamma+2s<1+2s,
\qquad \gamma+2s\notin\mathbb N.
\]
One last application of
\cite[Proposition 2.2]{RosOtonSerra} gives
\(u\in C^{\gamma+2s}(K_0)\), and therefore
\(u\in C^\theta(K_0)\). This proves \eqref{eq:interior-bootstrap}.
\end{proof}

The interior bootstrap allows us to improve the regularity of the truncated
right-hand side and, in the relevant regimes, of the boundary quotient itself.

\begin{lemma}
\label{lem:automatic-boundary-regularity}
Assume the hypotheses of Theorem~\ref{thm:main}, except for
\eqref{eq:TC}, and suppose that \(\Omega\) is of class
\(C^{2,\alpha}\) for some \(\alpha\in(0,1)\). If either \(s>1/2\), or there
exists \(a>0\) such that \(f\) is constant on \([0,a]\), then there exist
\(\rho'>0\) and \(\varepsilon>0\) such that
\[
\frac{u}{\delta^s}\in C^{1,\varepsilon}(\overline{U_{\rho'}}).
\]
\end{lemma}

\begin{proof}
Let \(w=\chi u\) be the bounded truncation constructed in Proposition~\ref{prop:boundary-quotient}. Thus
\begin{equation}\label{eq:global-higher-regularity-equation}
(-\Delta)^s w=G\quad\text{in }\Omega,
\qquad
w=0\quad\text{in }\mathbb R^N\setminus\Omega,
\end{equation}
in the weak sense, with \(G\in L^\infty(\Omega)\); see
\eqref{eq:truncated-global-equation}. Moreover, applying the bounded-data estimate of Ros-Oton--Serra
\cite[Proposition 1.1]{RosOtonSerra} componentwise as above gives
\[
w\in C^s(\mathbb R^N).
\]
Choose \(\varepsilon>0\) as follows. If \(s>1/2\), take
\[
0<\varepsilon<\min\{\alpha,2s-1,1-s\}.
\]
If instead \(f\) is constant on \([0,a]\), take
\[
0<\varepsilon<\min\{\alpha,s,1-s\}.
\]
Set
\[
\beta:=1+\varepsilon,
\qquad
\sigma:=\beta-s=1+\varepsilon-s.
\]
Then the above choices of $\varepsilon$ ensure that
\begin{equation}\label{eq:automatic-exponents}
0<\sigma<1,
\qquad
\beta>s,
\qquad
\beta,\ \beta-s,\ \beta+s\notin\mathbb N.
\end{equation}
In the case \(s>1/2\), the additional restriction
\(\varepsilon<2s-1\) also gives
\begin{equation}\label{eq:sigma-less-than-s}
\sigma<s.
\end{equation}
We first prove that
\begin{equation}\label{eq:global-G-holder}
G\in C^\sigma(\overline\Omega).
\end{equation}
We use the three regions from Proposition~\ref{prop:boundary-quotient}.
In a sufficiently thin boundary collar \(U_\rho\), with
\(\overline{U_\rho}\cap\overline{B_{2r}}=\emptyset\), one has \(w=u\) and
\begin{equation}\label{eq:localized-higher-regularity}
G=f(u)+h,
\qquad
h(x)=k_{N,s}\int_{\mathbb R^N}
\frac{(1-\chi(y))u(y)}{|x-y|^{N+2s}}\,\dd y.
\end{equation}
The supports of \(1-\chi\) and \(U_\rho\) are separated, so
\(h\in C^\infty(\overline{U_\rho})\). If \(s>1/2\), then
\(u=w\in C^s\) in the collar, hence
\(f(u)\in C^s(\overline{U_\rho})\); by
\eqref{eq:sigma-less-than-s}, \(G\in C^\sigma(\overline{U_\rho})\).
If instead \(f\) is constant on \([0,a]\), Proposition~\ref{prop:boundary-quotient} gives \(u=\delta^s\psi\) with \(\psi\) bounded
near the boundary. Thus \(u\to0\) uniformly as \(\delta\to0\), and after
shrinking \(\rho\) one has \(0\le u\le a\) in \(U_\rho\). Hence \(f(u)\)
is constant there and \(G\in C^\infty(\overline{U_\rho})\).

In \(B_{r/2}\), since \(w=0\),
\[
G(x)=-k_{N,s}\int_{\mathbb R^N}
\frac{w(y)}{|x-y|^{N+2s}}\,\dd y.
\]
The support of \(w\) is separated from \(B_{r/2}\), and therefore
\begin{equation}\label{eq:G-smooth-near-pole}
G\in C^\infty(\overline{B_{r/2}}).
\end{equation}
It remains to consider
\[
K:=\overline{\Omega\setminus(U_\rho\cup B_{r/2})}
\Subset\Omega\setminus\{0\}.
\]
Choose \(\delta_*>0\) so small that
\[
0<\delta_*<s-\varepsilon,
\qquad
2s+\sigma+\delta_*\notin\mathbb N.
\]
Since
\[
2s+\sigma+\delta_*=1+s+\varepsilon+\delta_*<1+2s,
\]
Lemma~\ref{lem:interior-bootstrap} yields, on a neighborhood
\(V\Subset\Omega\setminus\{0\}\) of \(K\),
\[
u\in C^{2s+\sigma+\delta_*}(V).
\]
Let \(\vartheta\in C_c^\infty(V)\) be equal to one in a neighborhood of
\(K\). Since \(w=\chi u\) and \(\chi\) is smooth,
\[
\vartheta w\in C_c^{2s+\sigma+\delta_*}(\mathbb R^N).
\]
For \(x\) near \(K\), using \eqref{eq:truncated-global-equation}, the following relation holds pointwise
\[
G(x)=(-\Delta)^s(\vartheta w)(x)
     +(-\Delta)^s((1-\vartheta)w)(x).
\]
Indeed, the second term is smooth there because \((1-\vartheta)w\) is supported at
positive distance from \(K\). By the standard H\"older estimate for the
fractional Laplacian \cite[Proposition 2.1.7]{SilvestreObstacle},
\[
(-\Delta)^s:C_c^{2s+\tau}(\mathbb R^N)\longrightarrow C^\tau(\mathbb R^N)
\]
for the noninteger exponents considered here. Taking
\(\tau=\sigma+\delta_*\), we obtain
\[
(-\Delta)^s(\vartheta w)\in C^{\sigma+\delta_*}(\mathbb R^N)
\subset C^\sigma(\mathbb R^N).
\]
Thus \(G\in C^\sigma\) near \(K\). Together with the collar estimate and
\eqref{eq:G-smooth-near-pole}, a finite covering of \(\overline\Omega\)
proves \eqref{eq:global-G-holder}.

We now apply the higher-order boundary estimate to the global
equation. A bounded open set with \(C^{2,\alpha}\) boundary has finitely many
connected components,
\[
\Omega=\bigcup_{j=1}^m\Omega_j,
\]
whose closures are pairwise separated by a positive distance. For each component set
\[
w_j:=\mathbf1_{\Omega_j}w.
\]
For \(x\in\Omega_j\), since \(w_j(x)=w(x)\) and \(w=0\) in \(\mathbb R^N\setminus\Omega\), direct subtraction of the two fractional Laplacians gives
\begin{equation}\label{eq:component-equation}
(-\Delta)^s w_j=G_j,
\qquad
G_j(x):=G(x)+k_{N,s}\int_{\Omega\setminus\Omega_j}
\frac{w(y)}{|x-y|^{N+2s}}\,\dd y.
\end{equation}
The interaction term is \(C^\infty(\overline{\Omega_j})\), since the other
components are at positive distance from \(\overline{\Omega_j}\). Hence
\(G_j\in C^{\beta-s}(\overline{\Omega_j})\). Moreover,
\(\beta+1=2+\varepsilon\) and \(\varepsilon<\alpha\), so every
\(C^{2,\alpha}\) component is, in particular, a \(C^{\beta+1}\) domain. Together with
\eqref{eq:automatic-exponents}, these are precisely the hypotheses of
\cite[Theorem 1.4]{AbatangeloRosOton} on each component: if \(d_j\) is the
regularized distance for \(\Omega_j\), then
\[
\frac{w_j}{d_j^s}\in C^\beta(\overline{\Omega_j})
=C^{1,\varepsilon}(\overline{\Omega_j}).
\]
For every connected component \(\Omega_j\) of \(\Omega\), one has
\(\partial\Omega_j\subset\partial\Omega\). Hence, for every \(x\in\Omega_j\),
\[
\dist(x,\partial\Omega_j)
=
\dist(x,\mathbb R^N\setminus\Omega_j)
=
\dist(x,\mathbb R^N\setminus\Omega)
=
\delta(x).
\]
Moreover, \(w_j=u\) in the boundary collar of \(\Omega_j\).
Applying \eqref{eq:regularized-distance-ratio} componentwise and taking the minimum of
the finitely many collar radii gives some \(\rho'>0\) such that
\[
\frac{u}{\delta^s}
=
\frac{w_j}{d_j^s}\left(\frac{d_j}{\delta}\right)^s
\in C^{1,\varepsilon}(\overline{U_{\rho'}}).
\]
\end{proof}

The preceding boundary estimate now yields the tangential cancellation required
in the main symmetry theorem.

\begin{proposition}
\label{prop:automatic-TC}
Assume the hypotheses of Theorem~\ref{thm:main}, except for
\eqref{eq:TC}, and suppose that \(\Omega\) is of class
\(C^{2,\alpha}\) for some \(\alpha\in(0,1)\). If either \(s>1/2\), or there
exists \(a>0\) such that \(f\) is constant on \([0,a]\), then
\eqref{eq:TC} holds.
\end{proposition}

\begin{proof}
By Lemma~\ref{lem:automatic-boundary-regularity}, for some
\(\rho'>0\) and \(\varepsilon>0\),
\[
\frac{u}{\delta^s}\in C^{1,\varepsilon}(\overline{U_{\rho'}}).
\]
For \(N\ge2\), the
overdetermined datum and continuity of the quotient give
\[
\frac{u}{\delta^s}=-c
\qquad\text{on }\partial\Omega.
\]
Hence all tangential derivatives of \(u/\delta^s\) vanish on the boundary.
For \(Q\in\partial\Omega\) and \(\tau\in T_Q\partial\Omega\) with
\(|\tau|=1\), the \(C^{1,\varepsilon}\)-expansion at \(Q\) yields
\[
\frac{u}{\delta^s}\bigl(Q+t(\nu_Q+\tau)\bigr)
-
\frac{u}{\delta^s}\bigl(Q+t(\nu_Q-\tau)\bigr)
=O(t^{1+\varepsilon})=o(t),
\]
since \(\varepsilon > 0\), which is \eqref{eq:TC}.
\end{proof}

\section{Comparison principles}
\label{sec:comparison-principles}

This section first records the local equation satisfied by the reflected
difference away from the two poles and then establishes the two comparison
ingredients that replace the standard maximum principle in the singular
context. The weak comparison lemma starts the inequality in small caps despite
the reflected pole, while the strong comparison lemma prevents an
antisymmetric supersolution from touching zero away from that pole.

\begin{lemma}\label{lem:away-from-poles}
Let \(\lambda<0\) satisfy \(\Omega'_\lambda\subset\Omega\), and let
\(A\subset\Omega_\lambda\) be a bounded Lipschitz open set such that
\[
\dist\bigl(\overline A,\{0,0^\lambda\}\bigr)>0.
\]
Then \(w_\lambda=u_\lambda-u\) belongs to
\(\mathcal D^s(A)\) and satisfies
\[
\mathcal E(w_\lambda,\phi)
=
\int_A\bigl(f(u_\lambda)-f(u)\bigr)\phi\,\dd x
\qquad\forall\phi\in H_0^s(A).
\]
In particular,
\[
(-\Delta)^s w_\lambda=c_\lambda(x)w_\lambda\quad\text{in }A,
\qquad c_\lambda\in L^\infty(A).
\]
\end{lemma}

\begin{proof}
Choose an open set \(V\) such that
\[
\overline A\subset V\Subset
\mathbb R^N\setminus\{0,0^\lambda\}.
\]
Then
\[
w_\lambda\in W^{s,2}(V)\cap L_s^1(\mathbb R^N).
\]
Since the reflection \(x\mapsto x^\lambda\) preserves Lebesgue measure and
\[
C_{\lambda}^{-1} (1+|x|^{N+2s}) \le 1+|x^{\lambda}|^{N+2s} \le C_{\lambda} (1+|x|^{N+2s}),
\]
a change of variables gives
\[
\|u_\lambda\|_{L_s^1}\le C_\lambda\|u\|_{L_s^1}.
\]
Let \(\theta\in C_c^\infty(V)\) equal one near
\(\overline A\), and write
\(w_\lambda=\theta w_\lambda+(1-\theta)w_\lambda\).  The first term belongs
to \(H^s(\mathbb R^N)\), so its energy pairing with
\(\phi\in H_0^s(A)\) is controlled by Cauchy--Schwarz.  The second term is
supported at positive distance from \(A\), and
\[
\sup_{x\in A}\int_{\mathbb R^N}
\frac{|(1-\theta(y))w_\lambda(y)|}{|x-y|^{N+2s}}\,\dd y<\infty
\]
by the local integrability and \(L_s^1\)-tail of \(w_\lambda\).  Hence
\[
|\mathcal E(w_\lambda,\phi)|
\leq C_A\|\phi\|_{H^s(\mathbb R^N)},
\]
which gives \(w_\lambda\in\mathcal D^s(A)\).

For \(\phi\in C_c^\infty(A)\), both \(\phi\) and its reflection are
admissible in the punctured weak equation. Indeed,
\(A^\lambda\subset\Omega'_\lambda\subset\Omega\), and the distance
assumption excludes both poles. Reflection invariance of the energy form and
subtraction of the two identities give
\[
\mathcal E(w_\lambda,\phi)
=
\int_A\bigl(f(u_\lambda)-f(u)\bigr)\phi\,\dd x.
\]
The identity extends to \( \phi \in H_0^s(A)\) by density.  Finally, \(u\) and
\(u_\lambda\) are bounded on \(\overline A\).  Defining
\[
c_\lambda(x)
:=
\begin{cases}
\dfrac{f(u_\lambda(x))-f(u(x))}
{u_\lambda(x)-u(x)},
&u_\lambda(x)\neq u(x),\\[1.2ex]
0,&u_\lambda(x)=u(x),
\end{cases}
\]
the local Lipschitz continuity of \(f\) gives
\(c_\lambda\in L^\infty(A)\), and the asserted linear equation follows.
\end{proof}

\begin{lemma}[Weak comparison principle in small domains]\label{lem:small-domain-comparison}
Let \(\lambda<0\) satisfy \(\Omega'_\lambda\subset\Omega\), and let
\(D\subset\Omega_\lambda\) be a bounded open set. Define \(q_\lambda\) almost everywhere by
\(q_\lambda:=\mathbf1_{\Sigma_\lambda}w_{\lambda}^{-}\), and assume that
\(w_\lambda\geq0\) a.e. in \(\Sigma_\lambda\setminus D\). Set
\[
M_\lambda:=\|u\|_{L^\infty(\Omega_\lambda)}<\infty.
\]
If \(L_{M_\lambda}\), the Lipschitz constant of \(f\) on \([0,M_\lambda]\), satisfies
\(L_{M_\lambda}<\lambda_1(D)\), then
\(w_\lambda\geq0\) in \(\Sigma_\lambda\setminus\{0^\lambda\}\).
\end{lemma}
\begin{proof}
For the sake of brevity, let us write
\[
q:=\mathbf1_{\Sigma_\lambda}w_{\lambda}^{-}.
\]
By the exterior nonnegativity assumption, \(q=0\) a.e. in
\(\mathbb R^N\setminus D\). The restriction to \(\Sigma_\lambda\) is essential. Indeed, by
antisymmetry,
\[
w_{\lambda}^{-}(x^\lambda)=w_{\lambda}^{+}(x),
\]
so the global negative part \(w_{\lambda}^{-}\) need not vanish outside
\(\Sigma_\lambda\) and therefore is not, in general, an admissible test
function for the equation in \(\Omega_\lambda\).

For \(\varepsilon<|\lambda|/4\), let
\[
\psi_\varepsilon:=\psi_{0^\lambda,\varepsilon}
\]
be the capacity cutoff from \eqref{eq:capacity-cutoffs}. Thus
\(\psi_\varepsilon=0\) near \(0^\lambda\),
\(\psi_\varepsilon\to1\) almost everywhere, and
\[
[1-\psi_\varepsilon]_{H^s(\mathbb R^N)}^2
=
[\eta_{0^\lambda,\varepsilon}]_{H^s(\mathbb R^N)}^2
\longrightarrow0.
\]
We first note that
\[
q\psi_\varepsilon,\qquad q\psi_\varepsilon^2
\in H_0^s(D).
\]
Indeed, set 
\[
\zeta_\varepsilon
:=
\psi_{0^\lambda,\varepsilon}\psi_{0,\varepsilon}.
\]
The cutoff \(\zeta_\varepsilon\) is invariant under reflection with
respect to \(T_\lambda\), removes both poles, and satisfies
\(\zeta_\varepsilon=\psi_\varepsilon\) in \(\Sigma_\lambda\).
Since \(w_\lambda\) has bounded support and belongs to
\(W^{s,2}_{\rm loc}(\mathbb R^N\setminus\{0,0^\lambda\})\), this gives
\begin{equation}\label{eq:cutoff-reflection-Hs}
\zeta_\varepsilon w_\lambda\in H^s(\mathbb R^N).
\end{equation}
Moreover, since \(q=0\) almost everywhere in
\(\mathbb R^N\setminus D\),
\(\zeta_\varepsilon w_\lambda\geq0\) in
\(\Sigma_\lambda\setminus D\). Hence the standard truncation property
for antisymmetric Sobolev functions
\cite[Proposition 3.1]{FallJarohs} yields
\[
\mathbf1_{\Sigma_\lambda}
(\zeta_\varepsilon w_\lambda)^-
=
q\psi_\varepsilon
\in H_0^s(D),
\]
and multiplication by \(\psi_\varepsilon\) still gives
\(q\psi_\varepsilon^2\in H_0^s(D)\).

The latter function is also an admissible test in the punctured weak
equations. Indeed, it is supported in
\(D\subset\Omega_\lambda\subset\Omega\), is separated from the original
pole, and vanishes near \(0^\lambda\); its reflection is supported in
\(\Omega'_\lambda\subset\Omega\) and vanishes near the original pole.
Using
\(H_0^s(\Omega)=\widetilde H^s(\Omega)\) and approximating by smooth
tests, we therefore subtract the weak equations for \(u_\lambda\) and \(u\).
This gives
\begin{equation}\label{eq:weak-difference-negative-part}
-\mathcal E(w_\lambda,q\psi_\varepsilon^2)
=
\int_D
\bigl(f(u)-f(u_\lambda)\bigr)
q\psi_\varepsilon^2\,\dd x.
\end{equation}
For completeness, set \(p:=\mathbf1_{\Sigma_\lambda}w_\lambda^+\) and, for
\(x,y\in\Sigma_\lambda\),
\[
K_1(x,y):=|x-y|^{-N-2s},
\qquad
K_2(x,y):=|x-y^\lambda|^{-N-2s}.
\]
Using antisymmetry and reflecting the integration over
\(\mathbb R^N\setminus\Sigma_\lambda\), one obtains, first for truncated
kernels and then by passage to the limit,
\[
\begin{aligned}
&-\mathcal E(w_\lambda,q\psi_\varepsilon^2)
-\mathcal E(q,q\psi_\varepsilon^2)\\
&\quad =k_{N,s}\int_{\Sigma_\lambda}\!\int_{\Sigma_\lambda}
q(x)\psi_\varepsilon(x)^2
\Bigl[p(y)\bigl(K_1(x,y)-K_2(x,y)\bigr)+q(y)K_2(x,y)\Bigr]
\,\dd y\,\dd x\geq0,
\end{aligned}
\]
since \(|x-y|<|x-y^\lambda|\) for \(x,y\in\Sigma_\lambda\). Hence
\begin{equation}\label{eq:antisymmetric-negative-part-energy}
-\mathcal E(w_\lambda,q\psi_\varepsilon^2)
\geq
\mathcal E(q,q\psi_\varepsilon^2).
\end{equation}
Using the elementary identity
\[
(a-b)(a\alpha^2-b\beta^2)
=
(a\alpha-b\beta)^2-ab(\alpha-\beta)^2, \quad \forall a,b,\alpha,\beta \in \R
\]
we obtain
\begin{equation}\label{eq:weighted-cutoff-energy}
\begin{aligned}
\mathcal E(q,q\psi_\varepsilon^2)
&=
\mathcal E(q\psi_\varepsilon,q\psi_\varepsilon)\\
&\quad-
\frac{k_{N,s}}2
\iint_{\mathbb R^N\times\mathbb R^N}
\frac{q(x)q(y)
\bigl(\psi_\varepsilon(x)-\psi_\varepsilon(y)\bigr)^2}
{|x-y|^{N+2s}}\,\dd x\,\dd y.
\end{aligned}
\end{equation}
On \(\{q>0\}\) one has
\[
q=u-u_\lambda,
\qquad
0\leq u_\lambda<u\leq M_\lambda,
\qquad
0\leq q\leq M_\lambda.
\]
Here \(M_\lambda<\infty\) because \(\lambda<0\) implies that
\(\overline{\Omega_\lambda}\) is a compact subset of
\(\overline\Omega\setminus\{0\}\), while
\(u\in C(\mathbb R^N\setminus\{0\})\). Therefore the last term is bounded in absolute value by
\[
CM_\lambda^2
[\eta_{0^\lambda,\varepsilon}]_{H^s(\mathbb R^N)}^2
=o(1).
\]
Since \(f\) is \(L_{M_\lambda}\)-Lipschitz on \([0,M_\lambda]\),
\eqref{eq:weak-difference-negative-part} consequently yields
\[
\mathcal E(q\psi_\varepsilon,q\psi_\varepsilon)
\leq
L_{M_\lambda}\int_D q^2\psi_\varepsilon^2\,\dd x+o(1).
\]
As \(q\psi_\varepsilon\in H_0^s(D)\), the variational characterization
of \(\lambda_1(D)\) gives
\[
\lambda_1(D)\int_D q^2\psi_\varepsilon^2\,\dd x
\leq
\mathcal E(q\psi_\varepsilon,q\psi_\varepsilon).
\]
Hence
\[
\bigl(\lambda_1(D)-L_{M_\lambda}\bigr)
\int_D q^2\psi_\varepsilon^2\,\dd x
\leq o(1).
\]
Since \(0\leq q\leq M_\lambda\), \(D\) is bounded, and
\(\psi_\varepsilon\to1\) almost everywhere, dominated convergence gives
\[
\bigl(\lambda_1(D)-L_{M_\lambda}\bigr)
\int_D q^2\,\dd x\leq0.
\]
By assumption \(\lambda_1(D)-L_{M_\lambda}>0\), and therefore \(q\equiv0\).
Thus \(w_\lambda\geq0\) almost everywhere in \(\Sigma_\lambda\).
Since \(w_\lambda\) is continuous away from \(0^\lambda\), the inequality
holds pointwise in
\(\Sigma_\lambda\setminus\{0^\lambda\}\).
\end{proof}

Once nonnegativity has been obtained, the next lemma upgrades it to strict
positivity away from the reflected pole, unless full symmetry has already
occurred.

\begin{lemma}[Strong maximum principle]\label{lem:strong-maximum}
Let \(\lambda<0\) satisfy \(\Omega'_\lambda\subset\Omega\), and assume that
\(w_\lambda\geq0\) in \(\Sigma_\lambda\setminus\{0^\lambda\}\). Then either
\(w_\lambda\equiv0\) in \(\Sigma_\lambda\setminus\{0^\lambda\}\), or
\(w_\lambda>0\) in \(\Omega_\lambda\setminus\{0^\lambda\}\).
\end{lemma}

\begin{proof}
Fix \(x_0\in\Omega_\lambda\setminus\{0^\lambda\}\), and choose a ball
\[
B:=B_r(x_0)\Subset\Omega_\lambda\setminus\{0^\lambda\}.
\]
Since \(B\) is at positive distance from both poles,
Lemma~\ref{lem:away-from-poles} gives
\[
w_\lambda\in\mathcal D^s(B),
\qquad
(-\Delta)^s w_\lambda=c_\lambda(x)w_\lambda
\quad\text{in }B,
\qquad c_\lambda\in L^\infty(B),
\]
in the weak sense on \(H_0^s(B)\). Together with antisymmetry with respect
to \(T_\lambda\), the bounded coefficient, and nonnegativity in
\(\Sigma_\lambda\setminus B\), these are precisely the hypotheses needed
for the strong maximum principle \cite[Corollary 3.4]{FallJarohs}.
Hence either \(w_\lambda\equiv0\) in
\(\Sigma_\lambda\setminus\{0^\lambda\}\), or
\(w_\lambda(x_0)>0\). Since \(x_0\) was arbitrary, the conclusion follows.
\end{proof}

\section{The moving plane argument}
\label{sec:moving-plane}

We now apply the moving-plane method. The proof follows the classical three-step structure: first we start the procedure, then we continue the plane up to the first geometrically critical position, and finally we show that the critical plane passes through the singular point. Repeating the argument in every direction forces the domain to be centered at the pole.

We develop the argument first in the direction \(e_1\), and set
\[
a:=\inf_{x\in\Omega}x_1.
\]
We define the first geometrically critical position by
\[
\bar\lambda
:=
\sup\left\{
\ell\in(a,0]:
\Omega'_\mu\subset\Omega
\ \text{for every }\mu\in(a,\ell)
\right\}.
\]
For \(N\geq2\), the standard geometric alternative for a \(C^2\) open set
states that, if \(\bar\lambda<0\), then at \(\bar\lambda\) either an internal
tangency occurs away from the plane or \(T_{\bar\lambda}\) is orthogonal to
\(\partial\Omega\); see Serrin's original moving-plane argument
\cite{Serrin}. When \(N=1\), there is no tangential direction and hence no orthogonality alternative; only endpoint contact can occur.  Define
\[
\Lambda
:=
\left\{
\lambda\in(a,\bar\lambda):
w_\mu\ge0
\text{ in }\Omega_\mu\setminus\{0^\mu\}
\text{ for every }\mu\in(a,\lambda]
\right\},
\]
and
\[
\lambda^*:=\sup\Lambda.
\]

Before starting the plane, we isolate how non-removability rules out a symmetry position reached before the plane meets the pole.

\begin{lemma}
\label{lem:symmetry-removability}
Let \(\lambda<0\) satisfy \(\Omega'_\lambda\subset\Omega\). If
\(w_\lambda\equiv0\) in \(\Sigma_\lambda\setminus\{0^\lambda\}\), then the singularity of \(u\) at
the origin is removable in the sense of Definition~\ref{def:removable-singularity}.
\end{lemma}

\begin{proof}
Choose \(r>0\) such that
\[
B_r(0)\Subset\Omega,
\qquad
B_r(0^\lambda)\Subset\Sigma_\lambda.
\]
For \(y\in B_r(0)\setminus\{0\}\), set \(x:=y^\lambda\). Since reflection
is an isometry,
\[
x\in B_r(0^\lambda)\setminus\{0^\lambda\}
\subset\Sigma_\lambda\setminus\{0^\lambda\}.
\]
Hence, by the hypothesis of the lemma,
\[
0=w_\lambda(x)=u(x^\lambda)-u(x)=u(y)-u(y^\lambda),
\]
and therefore \(u(y)=u(y^\lambda)\).
Since, for \(\lambda < 0\), \(0^\lambda\neq0\) and
\[
u\in W^{s,2}_{\loc}(\mathbb R^N\setminus\{0\})
\cap C(\mathbb R^N\setminus\{0\}),
\]
we have
\[
u\in W^{s,2}(B_r(0^\lambda))\cap C(B_r(0^\lambda)).
\]
As reflection maps \(B_r(0)\) isometrically onto \(B_r(0^\lambda)\), it follows that
\[
y\mapsto u(y^\lambda)\in W^{s,2}(B_r(0))\cap C(B_r(0)).
\]
It therefore defines a value at the origin and agrees almost everywhere with
\(u\) on the punctured ball. Assigning this value to \(u(0)\) gives an
extension \(\widetilde u\) such that
\[
\widetilde u\in W^{s,2}_{\loc}(\mathbb R^N)
\cap L^1_s(\mathbb R^N)\cap C(\mathbb R^N).
\]
In particular, \(\widetilde u\) and \(f(\widetilde u)\) are bounded near the
origin.

It remains to extend the equation. Let \(\phi\in C_c^\infty(\Omega)\), and let
\(\eta_\varepsilon:=\eta_{0,\varepsilon}\) and
\(\psi_\varepsilon:=1-\eta_\varepsilon\) be the capacity cutoffs in
\eqref{eq:capacity-cutoffs}. Then
\(\phi\psi_\varepsilon\in C_c^\infty(\Omega\setminus\{0\})\), so
\[
\mathcal E(\widetilde u,\phi\psi_\varepsilon)
=
\int_\Omega f(\widetilde u)\phi\psi_\varepsilon\,\dd x.
\]
The construction gives
\([\eta_\varepsilon]_{H^s(\mathbb R^N)}\to0\) and
\(\|\eta_\varepsilon\|_{L^2}\to0\). Multiplication by the fixed smooth
function \(\phi\) therefore yields
\[
\|\phi\eta_\varepsilon\|_{H^s(\mathbb R^N)}\longrightarrow0.
\]
Arguing as in the local-tail decomposition used in the proof of
Lemma~\ref{lem:away-from-poles}, we split the bilinear form into a fixed
compact neighborhood of \(\supp\phi\) and its complement. Local
\(H^s\)-regularity of \(\widetilde u\) controls the first part, whereas the \(L_s^1\)-tail controls the second.
Consequently,
\[
\mathcal E(\widetilde u,\phi\psi_\varepsilon)
\longrightarrow
\mathcal E(\widetilde u,\phi).
\]
The right-hand side converges by dominated convergence. Hence
\[
\mathcal E(\widetilde u,\phi)
=
\int_\Omega f(\widetilde u)\phi\,\dd x,
\]
which proves removability.
\end{proof}

Consequently, non-removability yields the exclusion used below:
\begin{equation}\label{eq:no-premature-symmetry}
\lambda<0
\quad\Longrightarrow\quad
w_\lambda\not\equiv0\ \text{in }\Sigma_\lambda\setminus\{0^\lambda\}.
\end{equation}

With this notation, the first step is to start the plane from the left,
where the cap is small and separated from both singular points.

\begin{proposition}[The plane starts]\label{prop:plane-starts}
The set \(\Lambda\) is nonempty.
\end{proposition}

\begin{proof}
We prove that, when \(\lambda\) is sufficiently close to \(a\), the cap
\(\Omega_\lambda\) is small, its reflection \(\Omega'_\lambda\) is contained
in \(\Omega\), and neither pole belongs to the cap.  Since \(0\in\Omega\), we
have \(a<0\).  The standard initial-position geometry for a \(C^2\) open set
gives \(\sigma_1>0\) such that
\[
\Omega'_\lambda\subset\Omega
\qquad\text{for every }\lambda\in(a,a+\sigma_1).
\]
Choose now
\(\sigma_0\in(0,\sigma_1)\) so small that
\[
a+\sigma_0<\frac a2<0.
\]
Then, for every \(\lambda\in(a,a+\sigma_0)\), one has
\[
\lambda<0
\]
and hence \(0\notin\Omega_\lambda\). Moreover,
\[
(0^\lambda)_1=2\lambda<a,
\]
so that \(0^\lambda\notin\Omega\), by the definition of \(a\). In particular,
\[
0\notin\Omega_\lambda,
\qquad
0^\lambda\notin\Omega_\lambda.
\]
Thus, in the initial position of the moving plane, neither \(u\) nor
\(u_\lambda\) has a singularity in \(\Omega_\lambda\). Indeed, for
\(\lambda\in(a,a+\sigma_0)\),
\[
\dist(\Omega_\lambda,0)\geq |a|/2,
\]
while, since \(x^\lambda=0\) if and only if \(x=0^\lambda\),
\[
|x^\lambda|
=
|x-0^\lambda|
\ge |x_1-2\lambda|
=
x_1-2\lambda
\ge a-2(a+\sigma_0)>0
\qquad
\text{for }x\in\overline\Omega_\lambda.
\]
Thus the points \(x\) and \(x^\lambda\) range in a fixed compact subset of
\(\mathbb R^N\setminus\{0\}\).  By continuity of \(u\), after decreasing
\(\sigma_0\) if necessary there is \(M_0>0\), independent of
\(\lambda\in(a,a+\sigma_0)\), such that
\[
0\leq u,u_\lambda\leq M_0
\qquad\text{in }\Omega_\lambda.
\]
In \(\Omega_\lambda\),
\[
(-\Delta)^s w_\lambda=f(u_\lambda)-f(u)=c_\lambda(x)w_\lambda,
\]
where \(c_\lambda\in L^\infty(\Omega_\lambda)\), uniformly for \(\lambda\) close to \(a\). Since
\[
|\Omega_\lambda|\to0\quad\text{as }\lambda\to a,
\]
by \eqref{eq:eigenvalue-measure-bound} we have
\[
\lambda_1(\Omega_\lambda)\ge \alpha_{N,s}|\Omega_\lambda|^{-\frac{2s}{N}}\to+\infty.
\]
Since \(M_\lambda\leq M_0\), we have
\(L_{M_\lambda}\leq L_{M_0}\). Therefore, after taking
\(\lambda\) close enough to \(a\), we obtain
\[
L_{M_\lambda}<\lambda_1(\Omega_\lambda).
\]
For \(x\in\Sigma_\lambda\setminus\Omega_\lambda\), one has \(u(x)=0\)
and \(u_\lambda(x)\geq0\).  Hence the negative set of \(w_\lambda\) in
\(\Sigma_\lambda\) is contained in \(\Omega_\lambda\).  Lemma~\ref{lem:small-domain-comparison}, applied with
\(D=\Omega_\lambda\), gives
\[
w_\lambda\ge0\quad\text{in }\Sigma_\lambda\setminus\{0^\lambda\}.
\]
By applying Lemma~\ref{lem:strong-maximum}, either
\(w_\lambda\equiv0\) in \(\Sigma_\lambda\setminus\{0^\lambda\}\), or
\(w_\lambda>0\) in \(\Omega_\lambda\setminus\{0^\lambda\}\).  The first
alternative cannot occur by \eqref{eq:no-premature-symmetry}.
Hence \((a,a+\sigma)\subset\Lambda\) for some \(\sigma>0\).
\end{proof}

Having started the plane, we next show that the comparison can be continued up to the first geometrically critical position.

\begin{proposition}[Continuation up to the critical position]\label{prop:continuation}
One has
\[
\lambda^*=\bar\lambda.
\]
\end{proposition}

\begin{proof}
Suppose, by contradiction, that \(\lambda^*<\bar\lambda\). Choose
\(\sigma_0>0\) so small that
\[
\lambda^*+\sigma_0<\min\{\bar\lambda,0\}.
\]
Let \(\mu_n\in\Lambda\) increase to \(\lambda^*\).  For every
\(x\in\Omega_{\lambda^*}\setminus\{0^{\lambda^*}\}\), one has
\(x\in\Omega_{\mu_n}\) for all large \(n\), and the continuity of \(u\)
away from the two poles gives
\[
w_{\lambda^*}(x)=\lim_{n\to\infty}w_{\mu_n}(x)\geq0.
\]
Thus
\[
w_{\lambda^*}\ge0
\quad\text{in }\Omega_{\lambda^*}\setminus\{0^{\lambda^*}\}.
\]
By Lemma~\ref{lem:strong-maximum}, either \(w_{\lambda^*}\equiv0\) in
\(\Sigma_{\lambda^*}\setminus\{0^{\lambda^*}\}\), or
\[
w_{\lambda^*}>0
\quad\text{in }\Omega_{\lambda^*}\setminus\{0^{\lambda^*}\}.
\]
The first alternative cannot occur by
\eqref{eq:no-premature-symmetry}. Hence the strict inequality
holds.

For every plane
\(\lambda\in[\lambda^*,\lambda^*+\sigma_0]\), one has
\(\Omega_\lambda\subset
\Omega\cap\{x_1\leq\lambda^*+\sigma_0\}\).  The closure of this set is a
compact subset of \(\overline\Omega\setminus\{0\}\).  Therefore there exists
\(M>0\), independent of \(\lambda\) in this interval, such that
\[
0\leq u\leq M
\qquad\text{in }\Omega_\lambda.
\]
In particular, \(M_\lambda\leq M\) and hence
\(L_{M_\lambda}\leq L_M\).
Fix \(\varepsilon_0>0\) so small that
\begin{equation}\label{eq:small-continuation-measure}
L_M<\alpha_{N,s}(2\varepsilon_0)^{-2s/N}.
\end{equation}
Choose a compact set
\[
K\Subset\Omega_{\lambda^*}\setminus\{0^{\lambda^*}\}
\]
such that
\[
|\Omega_{\lambda^*}\setminus K|<\varepsilon_0.
\]
Since \(w_{\lambda^*}\) is continuous and strictly positive on \(K\), it
has a positive minimum there.  Since \(K\Subset\Omega_{\lambda^*}\setminus\{0^{\lambda^*}\}\),
the reflected sets \(K^\lambda\) stay uniformly away from the origin
for \(\lambda\) sufficiently close to \(\lambda^*\). Hence \(u\) is
uniformly continuous on a fixed compact set containing all such
\(K^\lambda\), and therefore
\[
\sup_{x\in K}
|w_\lambda(x)-w_{\lambda^*}(x)|
\longrightarrow0
\qquad\text{as }\lambda\to\lambda^*.
\] Uniform continuity therefore gives
\(\bar\sigma\in(0,\sigma_0)\) such that, for every
\(\sigma\in(0,\bar\sigma)\),
\[
K\subset\Omega_{\lambda^*+\sigma},
\qquad
w_{\lambda^*+\sigma}>0
\quad\text{in a neighborhood of }K.
\]
Define \(q_\sigma\) almost everywhere by
\[
q_\sigma
:=
\mathbf1_{\Sigma_{\lambda^*+\sigma}}
w_{\lambda^*+\sigma}^{-}.
\]
Then
\[
q_\sigma=0
\quad\text{a.e. in }\mathbb R^N\setminus D_\sigma,
\qquad
D_\sigma:=\Omega_{\lambda^*+\sigma}\setminus K.
\]
Furthermore,
\[
|D_\sigma|
\leq
|\Omega_{\lambda^*}\setminus K|
+|\Omega_{\lambda^*+\sigma}\setminus\Omega_{\lambda^*}|
<2\varepsilon_0
\]
after decreasing \(\bar\sigma\), since the measure of the last strip tends
to zero with \(\sigma\).  Hence \eqref{eq:eigenvalue-measure-bound}
and \eqref{eq:small-continuation-measure} give
\[
L_{M_{\lambda^*+\sigma}}\leq L_M<\lambda_1(D_\sigma),
\]
uniformly for \(0<\sigma<\bar\sigma\).  Lemma~\ref{lem:small-domain-comparison} now yields
\[
q_\sigma\equiv0
\qquad\text{for every }\sigma\in(0,\bar\sigma).
\]
Thus
\[
w_{\lambda^*+\sigma}\ge0
\quad\text{in }\Sigma_{\lambda^*+\sigma}
\setminus\{0^{\lambda^*+\sigma}\}
\]
for every such \(\sigma\).  It follows that
\(\lambda^*+\bar\sigma/2\in\Lambda\), contradicting the definition of
\(\lambda^*\). Hence \(\lambda^*=\bar\lambda\).
\end{proof}

The same limiting argument used in the proof of Proposition~\ref{prop:continuation}, now with \(\lambda^*=\bar\lambda\), gives
\[
w_{\bar\lambda}\geq0
\quad\text{in }\Omega_{\bar\lambda}\setminus\{0^{\bar\lambda}\}.
\]
For \(x\in\Sigma_{\bar\lambda}\setminus\Omega_{\bar\lambda}\), the exterior
condition gives \(u(x)=0\) and \(u(x^{\bar\lambda})\geq0\).  Therefore
\begin{equation}\label{eq:critical-halfspace-sign}
w_{\bar\lambda}\geq0
\quad\text{in }\Sigma_{\bar\lambda}\setminus\{0^{\bar\lambda}\}.
\end{equation}
If \(\bar\lambda<0\), the alternative
\(w_{\bar\lambda}\equiv0\) in
\(\Sigma_{\bar\lambda}\setminus\{0^{\bar\lambda}\}\) is excluded by
\eqref{eq:no-premature-symmetry}.  Hence Lemma~\ref{lem:strong-maximum} yields
\[
w_{\bar\lambda}>0
\quad\text{in }\Omega_{\bar\lambda}\setminus\{0^{\bar\lambda}\}.
\]

To exclude the two critical-contact alternatives, we use the Hopf lemma and
the corner lemma of Fall--Jarohs; see, respectively,
\cite[Proposition 3.3 and Lemma 4.4]{FallJarohs}.  Our reflected difference
is not globally regular because of the two isolated poles.  This causes no
difficulty at a contact point separated from them: Lemma~\ref{lem:away-from-poles} places \(w_{\bar\lambda}\) in the local class
\(\mathcal D^s\) required in \cite{FallJarohs} and provides there the weak
equation with a bounded coefficient.  The half-space sign and the geometric
hypotheses are verified explicitly in the two cases below.

We first rule out the internal tangency alternative by combining the fractional
Hopf lemma with the constant overdetermined datum.

\begin{lemma}[Exclusion of internal tangency]\label{lem:no-tangency}
At the critical position \(\bar\lambda<0\), internal tangency of
\(\Omega'_{\bar\lambda}\) with \(\partial\Omega\) cannot occur.
\end{lemma}

\begin{proof}
Assume that internal tangency occurs at a point
\(P\in\partial\Omega\cap\Sigma_{\bar\lambda}\), so that
\(P^{\bar\lambda}\in\partial\Omega\).  Since \(\bar\lambda<0\), one has
\(P\neq0\).  Moreover, \(P\neq0^{\bar\lambda}\), because otherwise
\(P^{\bar\lambda}=0\in\Omega\) could not be a boundary point.  By
Proposition~\ref{prop:continuation} and Lemma~\ref{lem:strong-maximum}, we obtain
\[
w_{\bar\lambda}>0
\quad\text{in }\Omega_{\bar\lambda}\setminus\{0^{\bar\lambda}\}.
\]
Since
\[
u(P)=0,\qquad u_{\bar\lambda}(P)=u(P^{\bar\lambda})=0,
\]
we have \(w_{\bar\lambda}(P)=0\).

Let \(\eta\) denote the exterior normal at \(P\). Since \(P\in\Sigma_{\bar\lambda}\) and is separated from both poles, we may
choose \(r>0\) so small that the two nested tangent balls
\[
A:=B_r(P-r\eta),
\qquad
U:=B_{2r}(P-2r\eta)
\]
satisfy
\[
A\subset U \subset\Omega_{\bar\lambda},
\qquad
\overline{U}\Subset\Sigma_{\bar\lambda}
\setminus\{0,0^{\bar\lambda}\}.
\]
The coefficient \(c_{\bar\lambda}\) is bounded on a fixed neighborhood of
\(P\).  Since
\(\lambda_1(A)=r^{-2s}\lambda_1(B_1)\), decreasing \(r\) if necessary gives
\[
\|c_{\bar\lambda}\|_{L^\infty(U)}
<\lambda_1(A).
\]
We also take \(r\) small enough that
\(\Omega_{\bar\lambda}\setminus\overline A\) contains a nonempty open set.
The strict positivity and continuity of \(w_{\bar\lambda}\) then provide a
compact set
\[
K\Subset\Omega_{\bar\lambda}\setminus\overline A,
\qquad |K|>0,
\]
such that
\[
\operatorname*{ess\,inf}_K w_{\bar\lambda}>0.
\]
Lemma~\ref{lem:away-from-poles} gives
\(w_{\bar\lambda}\in\mathcal D^s(U)\) and the equation
\[
(-\Delta)^s w_{\bar\lambda}
=c_{\bar\lambda}(x)w_{\bar\lambda}
\quad\text{in }U,
\]
while \eqref{eq:critical-halfspace-sign} gives nonnegativity on
\(\Sigma_{\bar\lambda}\setminus\{0^{\bar\lambda}\}\), and hence almost
everywhere in \(\Sigma_{\bar\lambda}\).  Thus all the hypotheses of
\cite[Proposition 3.3]{FallJarohs} are satisfied with equation domain
\(U\) and interior ball \(A\).  It follows that
\[
(\partial_\eta)^s w_{\bar\lambda}(P)<0.
\]
On the other hand, reflection maps \(\eta\) to the exterior normal at
\(P^{\bar\lambda}\).  By the overdetermined condition,
\[
\begin{aligned}
(\partial_\eta)^s w_{\bar\lambda}(P)
&=-\lim_{t\to0^+}
\frac{u((P-t\eta)^{\bar\lambda})-u(P-t\eta)}{t^s}\\
&=c-c=0,
\end{aligned}
\]
which is a contradiction.
\end{proof}

It remains to exclude orthogonal contact; here the corner lower bound is
contradicted by the tangential cancellation \eqref{eq:TC}.

\begin{lemma}[Exclusion of the corner case]\label{lem:no-corner}
Assume \(N\geq2\).  At the critical position \(\bar\lambda<0\), the moving
plane \(T_{\bar\lambda}\) cannot be orthogonal to \(\partial\Omega\) at a
point of \(T_{\bar\lambda}\cap\partial\Omega\).
\end{lemma}

\begin{proof}
Suppose, by contradiction, that the orthogonality case occurs at
\(Q\in T_{\bar\lambda}\cap\partial\Omega\). Since \(Q\in\partial\Omega\)
and \(0\in\Omega\), we have \(Q\neq0\). Moreover, since
\(Q\in T_{\bar\lambda}\),
\[
Q^{\bar\lambda}=Q.
\]
Consequently,
\[
|Q-0^{\bar\lambda}|
=
|Q^{\bar\lambda}-0^{\bar\lambda}|
=
|Q-0|>0.
\]
Thus \(Q\) is separated from both the original pole and the reflected pole.

Since the orthogonality condition implies that \(e_1\) is tangent to
\(\partial\Omega\) at \(Q\), after a rotation fixing the \(e_1\)-axis we
may assume that
\[
\nu_Q=e_2,
\]
where \(\nu_Q\) is the inward unit normal at \(Q\). Set
\[
\bar\eta:=-e_1+e_2.
\]
Since \(\partial\Omega\) is of class \(C^2\) and \(Q\) is separated from
both poles, we may choose \(R>0\) so small that
\[
D:=B_R(Q+Re_2)\subset\Omega,
\qquad
\dist\bigl(\overline D,\{0,0^{\bar\lambda}\}\bigr)>0.
\]
The center \(Q+Re_2\) lies on \(T_{\bar\lambda}\), so \(D\) is symmetric
with respect to \(T_{\bar\lambda}\). Hence
\[
D^*:=D\cap\Sigma_{\bar\lambda}
\subset\Omega_{\bar\lambda}.
\]
Lemma~\ref{lem:away-from-poles} gives
\(w_{\bar\lambda}\in\mathcal D^s(D^*)\), and the strong maximum principle
gives
\[
w_{\bar\lambda}>0
\qquad\text{in }D^*.
\]
The coefficient of its equation is bounded there, while
\eqref{eq:critical-halfspace-sign} gives
\(w_{\bar\lambda}\geq0\) away from the reflected pole in
\(\Sigma_{\bar\lambda}\), and hence almost everywhere in that half-space.
Thus, applying \cite[Lemma 4.4]{FallJarohs}, there exist
\(C>0\) and \(t_0>0\) such that
\begin{equation}\label{eq:corner-lower-bound}
w_{\bar\lambda}(Q+t\bar\eta)
\geq Ct^{1+s}
\qquad\text{for every }t\in(0,t_0).
\end{equation}
We now derive the corresponding upper estimate. Since
\(Q\in T_{\bar\lambda}\),
\[
(Q+t\bar\eta)^{\bar\lambda}
=
Q+t(e_1+e_2).
\]
Set
\[
\psi:=\frac{u}{\delta^s},
\qquad
\bar\delta(x):=\delta(x^{\bar\lambda}).
\]
Since \(e_1\in T_Q\partial\Omega\), the tangential cancellation condition
\eqref{eq:TC}, applied with \(\nu_Q=e_2\) and \(\tau=e_1\), gives
\begin{equation}\label{eq:corner-quotient-upper}
\psi\bigl(Q+t(e_1+e_2)\bigr)
-
\psi\bigl(Q+t(-e_1+e_2)\bigr)
=o(t).
\end{equation}
The distance function is of class \(C^2\) in a tubular neighborhood of
\(\partial\Omega\). Since
\[
\nabla\delta(Q)=e_2
\]
and \(|\nabla\delta|=1\) there, differentiation of
\(|\nabla\delta|^2=1\) gives
\[
D^2\delta(Q)e_2=0.
\]
Taylor's formula therefore yields
\[
\delta(Q+t\bar\eta)
=
t+\frac{t^2}{2}D^2\delta(Q)[e_1,e_1]+o(t^2)
\]
and
\[
\begin{aligned}
\bar\delta(Q+t\bar\eta)
&=
\delta\bigl(Q+t(e_1+e_2)\bigr)\\
&=
t+\frac{t^2}{2}D^2\delta(Q)[e_1,e_1]+o(t^2).
\end{aligned}
\]
Hence
\[
\bar\delta(Q+t\bar\eta)-\delta(Q+t\bar\eta)=o(t^2),
\qquad
\bar\delta(Q+t\bar\eta)\asymp
\delta(Q+t\bar\eta)\asymp t.
\]
By the mean value theorem, for some \(\xi_t\) between
\(\bar\delta(Q+t\bar\eta)\) and \(\delta(Q+t\bar\eta)\),
\[
\begin{aligned}
&\bar\delta(Q+t\bar\eta)^s
-\delta(Q+t\bar\eta)^s\\
&\qquad
=
s\xi_t^{\,s-1}
\bigl(
\bar\delta(Q+t\bar\eta)
-\delta(Q+t\bar\eta)
\bigr).
\end{aligned}
\]
Since \(\xi_t\asymp t\), it follows that
\[
s\xi_t^{\,s-1}
\bigl(
\bar\delta(Q+t\bar\eta)
-\delta(Q+t\bar\eta)
\bigr)
=
O(t^{s-1})o(t^2)
=
o(t^{1+s}),
\]
and therefore
\begin{equation}\label{eq:corner-distance-upper}
\bar\delta(Q+t\bar\eta)^s
-\delta(Q+t\bar\eta)^s
=
o(t^{1+s}).
\end{equation}
Moreover, Proposition~\ref{prop:boundary-quotient} makes \(\psi\) bounded
in the boundary collar, and
\[
\delta(Q+t\bar\eta)^s=O(t^s).
\]
Thus
\[
\begin{aligned}
w_{\bar\lambda}(Q+t\bar\eta)
&=
u\bigl(Q+t(e_1+e_2)\bigr)
-u\bigl(Q+t(-e_1+e_2)\bigr)\\
&=
\bar\delta(Q+t\bar\eta)^s
\psi\bigl(Q+t(e_1+e_2)\bigr)\\
&\quad-
\delta(Q+t\bar\eta)^s
\psi\bigl(Q+t(-e_1+e_2)\bigr)\\
&=
\bigl(
\bar\delta(Q+t\bar\eta)^s
-\delta(Q+t\bar\eta)^s
\bigr)
\psi\bigl(Q+t(e_1+e_2)\bigr)\\
&\quad+
\delta(Q+t\bar\eta)^s
\Bigl(
\psi\bigl(Q+t(e_1+e_2)\bigr)
-\psi\bigl(Q+t(-e_1+e_2)\bigr)
\Bigr).
\end{aligned}
\]
Together with \eqref{eq:corner-quotient-upper} and
\eqref{eq:corner-distance-upper}, this gives
\begin{equation}\label{eq:corner-upper-bound}
w_{\bar\lambda}(Q+t\bar\eta)
=
o(t^{1+s})
\qquad\text{as }t\to0^+.
\end{equation}
This contradicts \eqref{eq:corner-lower-bound}.
\end{proof}

Having excluded both possible critical-contact configurations, we can now complete the moving-plane argument and prove the main theorem.

\begin{proof}[Proof of Theorem~\ref{thm:main}]
Suppose, for contradiction, that \(\bar\lambda<0\). At the first critical
position either internal tangency or orthogonality must occur. Lemma~\ref{lem:no-tangency} excludes the first case and Lemma~\ref{lem:no-corner}
excludes the second. When \(N=1\), only the tangency alternative is present,
and Lemma~\ref{lem:no-tangency} already gives the contradiction. Hence
\(\bar\lambda=0\).

The argument was performed in the direction \(e_1\). It gives that
the critical plane in that direction is \(T_0=\{x_1=0\}\), and hence
\[
u(x)\le u(x^0)
\quad\text{for }x\in\Omega\cap\{x_1<0\},
\]
where \(x^0=(-x_1,x_2,\dots,x_N)\). Repeating the argument from the opposite
direction gives the reverse inequality. Thus both \(u\) and \(\Omega\) are
symmetric with respect to \(\{x_1=0\}\). The moving-plane inequalities and
Lemma~\ref{lem:strong-maximum} also show that \(u\) is strictly
increasing in the \(x_1\)-direction in \(\Omega\cap\{x_1<0\}\).
The same argument applies in every direction \(\nu\in\mathbb S^{N-1}\).
Therefore \(\Omega\) is symmetric with respect to every hyperplane passing
through \(0\). 

As in \cite{FallJarohs}, no connectedness assumption on \(\Omega\) is
needed. We give here a direct consequence of the moving-plane inclusions
which shows that connectedness follows a posteriori.
Fix a unit vector \(\nu\), write \(x=y+t\nu\) with
\(y\in\nu^\perp\), and set
\[
I_y:=\{t\in\mathbb R:y+t\nu\in\Omega\}.
\]
The conclusion \(\bar\lambda=0\) in the direction \(\nu\) gives the geometric
inclusion obtained during the moving-plane procedure for every plane
\(\{x\cdot\nu=\mu\}\) with \(\mu<0\). Applying the same argument in the
opposite direction \(-\nu\) gives the corresponding inclusion for every
\(\mu>0\). Moreover symmetry at the critical plane gives \(I_y=-I_y\).
If \(0\le t_1<t_2\) and \(t_2\in I_y\), choose
\(\mu=(t_1+t_2)/2>0\). The right-cap inclusion reflects the point
\(y+t_2\nu\) into \(y+t_1\nu\), and therefore \(t_1\in I_y\). By symmetry
the same holds on the negative side. Thus every nonempty line section \(I_y\)
is an interval. Since \(\nu\) is arbitrary, the intersection of \(\Omega\) with every line is an interval. Hence \(\Omega\) is convex and, in particular,
connected. Moreover, symmetry with respect to every hyperplane through the
origin makes \(\Omega\) invariant under all orthogonal transformations.
Since \(\Omega\) is bounded, open, convex, and contains the origin, it follows
that
\[
\Omega=B_R(0)
\]
for some \(R>0\).
\end{proof}

\section{The torsion case}
\label{sec:torsion}

In the torsion case \(f\equiv1\), subtracting the explicit torsion function leaves a nonnegative \(s\)-harmonic function in the punctured ball. The next lemma identifies the singular part as a multiple of the Green function.

\begin{lemma}[Identification of the singular Green part]
\label{lem:green-singular-part}
Let \(N>2s\), \(R>0\), and let
\(v\in L^1_s(\mathbb R^N)\cap
W^{s,2}_{\loc}(\mathbb R^N\setminus\{0\})\cap
C(\mathbb R^N\setminus\{0\})\) be nonnegative in \(\mathbb R^N\), vanish
in \(\mathbb R^N\setminus B_R\), and satisfy \((-\Delta)^s v=0\) weakly in
\(B_R\setminus\{0\}\). Assume moreover that the singularity at the origin is
non-removable in the sense of Definition~\ref{def:removable-singularity},
with \(f\equiv0\). Then there exists \(k>0\) such that
\[
v(x)=kG_R(x,0)
\qquad \forall x\in B_R\setminus\{0\}.
\]
\end{lemma}

\begin{proof}
The punctured weak equation implies
\[
(-\Delta)^s v=0
\quad\text{in }\mathcal D'(B_R\setminus\{0\}).
\]
Our space \(L_s^1(\mathbb R^N)\) is exactly the space denoted by
\(L_{2s}\) in Li--Wu--Xu \cite{LiWuXu}. After rescaling \(B_R\) onto the
unit ball, \(v\) is nonnegative, belongs to \(L_{2s}\), and satisfies the
punctured distributional equation with \(c\equiv0\) and \(f\equiv0\);
the standing assumption \(N>2s\) gives \(0<s<N/2\), the dimensional
range required there. Thus \cite[Theorem 4]{LiWuXu} yields a constant
\(k\geq0\) such that
\begin{equation}\label{eq:bocher-dirac}
(-\Delta)^s v=k\delta_0
\quad\text{in }\mathcal D'(B_R).
\end{equation}

We claim that \(k>0\). If \(k=0\), then \(v\) is distributionally
\(s\)-harmonic across the origin. The local regularity theorem for very weak
\(s\)-harmonic functions
\cite[Main Theorem, item (2)]{CarbottiCitoLaMannaPallara} gives a smooth
representative of \(v\) in a neighborhood of \(0\). Assigning to \(v(0)\)
the value of this representative and using the assumed local energy
regularity away from the origin, we obtain
\[
v\in W^{s,2}_{\loc}(\mathbb R^N)
\cap L_s^1(\mathbb R^N)
\cap C(\mathbb R^N).
\]
Moreover, the distributional identity \(( -\Delta)^s v=0\) in \(B_R\),
together with the integration-by-parts identity recalled in Section~\ref{sec:preliminaries}, gives
\[
\mathcal E(v,\phi)=0
\qquad\forall\phi\in C_c^\infty(B_R).
\]
Thus the origin would be removable in the sense of Definition~\ref{def:removable-singularity}, with \(f\equiv0\), a contradiction. Hence
\(k>0\).

Set
\[
z:=v-kG_R(\cdot,0).
\]
By \eqref{eq:bocher-dirac} and the normalization
\(( -\Delta)^sG_R(\cdot,0)=\delta_0\) in \(\mathcal D'(B_R)\),
\[
(-\Delta)^s z=0\quad\text{in }\mathcal D'(B_R),
\qquad
z=0\quad\text{in }\mathbb R^N\setminus B_R.
\]
The singular behavior \(G_R(x,0)=O(|x|^{2s-N})\) as \(x\to0\), together
with \(N>2s\), shows that \(G_R(\cdot,0)\in L_s^1(\mathbb R^N)\), and hence
\(z\in L_s^1(\mathbb R^N)\). The same very-weak regularity theorem therefore
gives
\[
z\in C^\infty_{\loc}(B_R).
\]

It remains to control the boundary behavior. Set
\(\delta_R(x):=R-|x|\). Proposition~\ref{prop:boundary-quotient}, applied
to \(v\) with right-hand side zero, yields
\[
|v(x)|\leq C\delta_R(x)^s
\qquad\text{for }x\text{ near }\partial B_R.
\]
On the other hand, the Boggio formula
\begin{equation}\label{eq:boggio}
G_R(x,0)
=
c_{N,s}|x|^{2s-N}
\int_0^{(R^2-|x|^2)/|x|^2}
\frac{t^{s-1}}{(1+t)^{N/2}}\,\dd t
\end{equation}
gives the same estimate for \(G_R(\cdot,0)\), because
\[
\frac{R^2-|x|^2}{|x|^2}=O(\delta_R(x))
\qquad\text{as }x\to\partial B_R.
\]
Consequently,
\[
|z(x)|\leq C\delta_R(x)^s
\qquad\text{near }\partial B_R,
\]
so \(z\) extends continuously by zero to all of \(\mathbb R^N\). Since
\(z\) is smooth in the ball, distributionally \(s\)-harmonic there, and
belongs to \(L_s^1(\mathbb R^N)\), it is also pointwise \(s\)-harmonic in
\(B_R\). Thus all the hypotheses of \cite[Theorem 2.10]{Bucur} are
satisfied, with zero exterior datum, and the uniqueness statement there
gives \(z\equiv0\). Therefore
\[
v(x)=kG_R(x,0)
\qquad\forall x\in B_R\setminus\{0\}.
\]
\end{proof}

We now turn to the second main result, which gives the explicit classification in the torsion case.

\begin{proof}[Proof of Corollary~\ref{cor:torsion-classification}]
Corollary~\ref{cor:automatic-TC}, item \textup{(ii)}, gives
\(\Omega=B_R(0)\). Let \(\tau_R\) be the torsion function of the ball and set
\[
v:=u-\tau_R.
\]
Then \(v\) is \(s\)-harmonic in \(B_R\setminus\{0\}\), vanishes outside the
ball, and belongs to the same local energy and tail classes as \(u\).  Its
singularity is non-removable: otherwise a removable \(s\)-harmonic extension
of \(v\), added to the regular torsion function \(\tau_R\), would give a
removable extension of \(u\) solving \((-\Delta)^s u=1\) in the whole ball.

We claim that \(v\geq0\).  On \(\{v<0\}\),
\[
v^-=\tau_R-u\leq\tau_R,
\]
so \(v^-\) is bounded and supported in \(\overline{B_R}\).  Let
\(\psi_\varepsilon=1-\eta_{0,\varepsilon}\) be the capacity cutoff from
\eqref{eq:capacity-cutoffs}, and set
\[
r_\varepsilon:=v^-\psi_\varepsilon.
\]
For each fixed \(\varepsilon>0\), the cutoff removes the only possible
singularity. More precisely, the Lipschitz truncation \(v^-\) belongs locally
to \(W^{s,2}\) away from the origin, and \(r_\varepsilon\) has bounded
support. The same localization argument used in
\eqref{eq:cutoff-reflection-Hs} therefore gives
\[
r_\varepsilon\in H^s(\mathbb R^N).
\]
Since it vanishes a.e. in \(\mathbb R^N\setminus B_R\), this is exactly
\(r_\varepsilon\in H_0^s(B_R)\). Multiplication by the smooth cutoff once
more gives \(v^-\psi_\varepsilon^2=r_\varepsilon\psi_\varepsilon\in
H_0^s(B_R)\). This function vanishes in a neighborhood of the origin; the density argument used in the proof of Lemma~\ref{lem:small-domain-comparison} therefore makes it an admissible test in
the punctured weak equation.  With \(q:=v^-\), the pointwise negative-part inequality underlying
\eqref{eq:antisymmetric-negative-part-energy} applies directly here and gives
\[
0=\mathcal E(v,q\psi_\varepsilon^2)
\leq-\mathcal E(q,q\psi_\varepsilon^2).
\]
Combining this inequality with the identity
\eqref{eq:weighted-cutoff-energy} yields
\[
\mathcal E(r_\varepsilon,r_\varepsilon)
\leq
C\|v^-\|_\infty^2
[\eta_{0,\varepsilon}]_{H^s(\mathbb R^N)}^2.
\]
The right-hand side tends to zero since \(N>2s\).  Since
\(r_\varepsilon\in H_0^s(B_R)\),
\[
\lambda_1(B_R)\|r_\varepsilon\|_{L^2(B_R)}^2
\leq
\mathcal E(r_\varepsilon,r_\varepsilon)
\longrightarrow0.
\]
On the other hand,
\(r_\varepsilon\to v^-\) in \(L^2(B_R)\) by dominated convergence.
Consequently \(v^-\equiv0\), and hence \(v\geq0\).  Lemma~\ref{lem:green-singular-part} now gives
\[
u(x)=\tau_R(x)+kG_R(x,0)
\qquad\text{with }k>0.
\]
It remains to identify \(k\). Writing \(\delta(x)=R-|x|\), the boundary
condition is
\[
-c=\lim_{x\to\partial B_R}\frac{u(x)}{\delta(x)^s}.
\]
The explicit torsion function satisfies
\[
\tau_R(x)=\gamma_{N,s}(R^2-|x|^2)^s_+,
\qquad
\lim_{x\to\partial B_R}\frac{\tau_R(x)}{\delta(x)^s}
=\gamma_{N,s}(2R)^s.
\]
For \(r=|x|\), set \(A_r=(R^2-r^2)/r^2\). Then
\[
\frac{A_r}{R-r}\longrightarrow\frac2R.
\]
Moreover, since \(A_r\to0\) and
\((1+t)^{-N/2}=1+O(t)\) as \(t\to0\),
\[
\int_0^{A_r}\frac{t^{s-1}}{(1+t)^{N/2}}\,\dd t
=\int_0^{A_r}t^{s-1}\bigl(1+O(t)\bigr)\,\dd t
=\frac{A_r^s}{s}\bigl(1+o(1)\bigr).
\]
The Boggio formula \eqref{eq:boggio} therefore gives
\[
\lim_{x\to\partial B_R}
\frac{G_R(x,0)}{\delta(x)^s}
=
\frac{c_{N,s}}{s}\,2^sR^{s-N}.
\]
Substitution in the boundary condition yields
\[
k=
\frac{sR^{N-s}}{c_{N,s}2^s}
\left(-c-\gamma_{N,s}(2R)^s\right),
\]
as claimed.
Since the singularity is non-removable, \(k>0\), and therefore
\[
-c>\gamma_{N,s}(2R)^s.
\]
If the family is enlarged to allow \(k=0\), equality is equivalent to
\(u=\tau_R\), and the origin is removable. Finally, the torsion term is
bounded near the origin, while in \eqref{eq:boggio} the upper integration
limit tends to infinity. Monotone convergence and \(N>2s\) give
\[
\int_0^{+\infty}\frac{t^{s-1}}{(1+t)^{N/2}}\,\dd t
=
B(s,N/2-s)
=\frac{\Gamma(s)\Gamma(N/2-s)}{\Gamma(N/2)}.
\]
Multiplying the decomposition by \(|x|^{N-2s}\) yields the asserted pole
limit and completes the proof.
\end{proof}

\appendix
\section{A solution for which \texorpdfstring{$(\mathrm{TC})$}{(TC)} holds but
\texorpdfstring{$u/\delta^s\notin C^1$}{u/delta s is not C1}}
\label{app:strict-TC-example}

We give the verification of the example announced in the Introduction. Fix
\(N=2\) and \(s\in(0,1/2)\). The range \(s>1/2\) is excluded because
Lemma~\ref{lem:automatic-boundary-regularity} gives
\(u/\delta^s\in C^{1,\varepsilon}\) there, so this separation from \(C^1\)-regularity cannot occur. Set
\[
\rho(x):=1-|x|^2.
\]
Since the pole is strictly interior, the Dirac mass at the origin is an
admissible measure datum. By the theory for measure data in
\cite{GomezCastroVazquezMeasure}, there exists a unique nonnegative very weak
solution \(u \in L^{1}_{s}(\R^{2})\) of
\begin{equation}\label{eq:appendix-resolvent}
\bigl((-\Delta)^s+I\bigr)u=\delta_0
\quad\text{in }B_1,
\qquad
u=0
\quad\text{in }\mathbb R^2\setminus B_1.
\end{equation}
Let \(G_1(\cdot,0)\) be the Green function of \((-\Delta)^s\) in \(B_1\).
Since
\[
\bigl((-\Delta)^s+I\bigr)G_1(\cdot,0)
=
\delta_0+G_1(\cdot,0)
\geq\delta_0,
\]
the comparison estimate in \cite[Theorem 1.1]{ChenVeron} gives
\[
0\leq u\leq G_1(\cdot,0).
\]
Thus Boggio's formula yields
\[
u(x)=O(\delta(x)^s)
\qquad\text{as }x\to\partial B_1.
\]
Away from the origin,
\[
(-\Delta)^s u=-u.
\]
Interior regularity and the strong minimum principle give
\(u>0\) in \(B_1\setminus\{0\}\), while uniqueness and rotational
invariance imply that \(u\) is radial. Moreover, the singularity cannot be
removable, since otherwise \((-\Delta)^su+u=0\) would hold distributionally
across the origin, contradicting \eqref{eq:appendix-resolvent}.

Apply now the radial cutoff construction used in the proof of
Proposition~\ref{prop:boundary-quotient}. Since \(u\) is smooth on the
transition annulus, we obtain a radial function \(v\), equal to \(u\) near
\(\partial B_1\), such that
\begin{equation}\label{eq:appendix-cutoff}
v\in H^s_0(B_1)\cap L^\infty(\mathbb R^2),
\qquad
(-\Delta)^s v=-v+H
\quad\text{in }B_1,
\qquad
H\in C^\infty(\overline{B_1}),
\end{equation}
with \(H\) radial. By \cite[Proposition 1.1]{RosOtonSerra},
\[
v\in C^s(\mathbb R^2).
\]
Choose \(\varepsilon\in(0,s)\) so that, with
\[
\beta_1:=2s-\varepsilon,
\qquad
\beta_2:=3s-\varepsilon,
\]
the noninteger conditions in \cite[Theorem 1.4]{AbatangeloRosOton} are
satisfied for both exponents. The function \(\rho\) is a regularized distance
in the sense used there. Since
\[
-v+H\in C^s(\overline{B_1})
\subset C^{s-\varepsilon}(\overline{B_1}),
\]
Theorem~1.4 of \cite{AbatangeloRosOton}, with \(\beta=\beta_1\), gives
\begin{equation}\label{eq:appendix-first-bootstrap}
\frac{v}{\rho^s}
\in C^{\beta_1}(\overline{B_1}).
\end{equation}
By radiality its boundary trace is a constant, say
\(\kappa_s:=\left.(v/\rho^s)\right|_{\partial B_1}\). Since \(v=u\) near the
boundary, the fractional Hopf lemma
\cite[Proposition 3.3 and Remark 3.5]{FallJarohs}, applied to \(u\) in an
interior tangent ball disjoint from the pole and chosen so small that
\(1<\lambda_1(B)\), gives \(\kappa_s>0\).

Since $v$ is radial and the fractional Laplacian is rotation invariant,
$H=(-\Delta)^sv+v$ is radial as well; hence
$h_0:=H|_{\partial B_1}$ is constant. We first remove this constant boundary
part. Dyda's formula \cite{Dyda} gives
\[
(-\Delta)^s\rho_+^s=C_s,
\qquad C_s=2^{2s}\Gamma(1+s)^2>0.
\]
Set $v^{(1)}:=v-(h_0/C_s)\rho_+^s$. Then
\[
(-\Delta)^s v^{(1)}
=-v+(H-h_0)
=-\kappa_s\rho^s
-\rho^s\left(\frac{v}{\rho^s}-\kappa_s\right)+(H-h_0).
\]
After cancelling the boundary value \(h_0\) of \(H\), the remaining term
\(H-h_0\) is \(O(\rho)\); hence \(-\kappa_s\rho^s\), which comes from
\(-v\), is the leading boundary term. Since $\beta_1=2s-\varepsilon>s$, this term is
not in $C^{\beta_1}(\overline{B_1})$ and prevents the second bootstrap. We next construct a second corrector whose fractional Laplacian cancels precisely
this term. For the power $3s$, Dyda's formula \cite{Dyda} gives
\[
(-\Delta)^s\rho_+^{3s}
=D_s\,{}_2F_1(1+s,-2s;1;|x|^2),
\qquad
D_s=2^{2s}\frac{\Gamma(1+s)\Gamma(1+3s)}{\Gamma(1+2s)}.
\]
The connection formula for ${}_2F_1$ near $z=1$, see
\cite[Corollary~2.3.3]{AndrewsAskeyRoy}, specialized to these parameters, reads
\[
{}_2F_1(1+s,-2s;1;z)
=c_{0,s}+\frac{\Gamma(-s)}{\Gamma(1+s)\Gamma(-2s)}(1-z)^s+r_s(z),
\]
where $c_{0,s}\in\mathbb R$, $r_s\in C^1([0,1])$, and $r_s(1)=0$.
Since $\rho=1-|x|^2$, we obtain
\begin{equation}\label{eq:appendix-three-s}
(-\Delta)^s\rho_+^{3s}=K_{0,s}+K_{1,s}\rho^s+R_s,
\end{equation}
where $K_{0,s}\in\mathbb R$, $R_s\in C^1(\overline{B_1})$,
$R_s=0$ on $\partial B_1$, and
\[
K_{1,s}
=D_s\frac{\Gamma(-s)}{\Gamma(1+s)\Gamma(-2s)}
=2^{2s+1}\frac{\Gamma(1-s)\Gamma(1+3s)}
{\Gamma(1-2s)\Gamma(1+2s)}>0.
\]
Set
\[
\zeta_s:=\frac{\kappa_s}{K_{1,s}}
\left(\rho_+^{3s}-\frac{K_{0,s}}{C_s}\rho_+^s\right),
\qquad
(-\Delta)^s\zeta_s=\kappa_s\rho^s+\frac{\kappa_s}{K_{1,s}}R_s.
\]
The $\rho^s$ term cancels the obstruction above, while the added
$\rho_+^s$ term cancels the constant $K_{0,s}$ produced by
$(-\Delta)^s\rho_+^{3s}$. Finally, with $v^{(2)} := v^{(1)}+\zeta_s =: \widetilde v$,
\[
(-\Delta)^s\widetilde v
=-\rho^s\left(\frac{v}{\rho^s}-\kappa_s\right)
+(H-h_0)+\frac{\kappa_s}{K_{1,s}}R_s.
\]
The first term belongs to \(C^{\beta_1}(\overline{B_1})\): indeed,
\(v/\rho^s-\kappa_s\in C^{\beta_1}\) and vanishes on the boundary, so
multiplication by \(\rho^s\) preserves this H\"older exponent. Since
\(H-h_0\) is smooth and vanishes on \(\partial B_1\), while
\(R_s\in C^1(\overline{B_1})\), it follows that
\[
(-\Delta)^s\widetilde v
\in C^{\beta_1}(\overline{B_1}).
\]
A second application of \cite[Theorem 1.4]{AbatangeloRosOton}, now with
\(\beta=\beta_2\), gives
\begin{equation}\label{eq:appendix-second-bootstrap}
\frac{\widetilde v}{\rho^s}
\in C^{\beta_2}(\overline{B_1}).
\end{equation}
Let \(Q\in\partial B_1\) and \(x_t=(1-t)Q\). Then, for \(t\) small
\[
\rho(x_t)=t(2-t)\asymp t.
\]
Since \(\beta_2>2s\) and \(2s<1\),
\eqref{eq:appendix-second-bootstrap} yields
\[
\frac{\widetilde v(x_t)}{\rho(x_t)^s}
=
\left.\frac{\widetilde v}{\rho^s}\right|_{\partial B_1}
+o(\rho(x_t)^{2s}).
\]
From the definition of \(\widetilde v\),
\[
\frac{v}{\rho^s}
=
\frac{h_0}{C_s}
-\frac{\kappa_s}{K_{1,s}}
\left(\rho^{2s}-\frac{K_{0,s}}{C_s}\right)
+\frac{\widetilde v}{\rho^s}.
\]
Taking the boundary limit and using the definition of \(\kappa_s\), the constant
terms combine to \(\kappa_s\). Therefore
\begin{equation}\label{eq:appendix-rho-expansion}
\frac{v}{\rho^s}
=
\kappa_s-\frac{\kappa_s}{K_{1,s}}\rho^{2s}+o(\rho^{2s}).
\end{equation}
Since
\[
\rho=(1+|x|)\delta
\]
and \(v=u\) near \(\partial B_1\), while \(2s<1\), we obtain
\begin{equation}\label{eq:appendix-delta-expansion}
\frac{u}{\delta^s}
=
A_s-B_s\delta^{2s}+o(\delta^{2s}),
\qquad
A_s:=2^s\kappa_s>0,
\quad
B_s:=\frac{2^{3s}\kappa_s}{K_{1,s}}>0.
\end{equation}
In particular, along \(x_t=(1-t)Q\), where
\(\delta(x_t)=t\),
\[
\frac{
\dfrac{u(x_t)}{\delta(x_t)^s}-A_s
}{t}
=
-B_s t^{2s-1}+o(t^{2s-1}),
\]
and hence \(u/\delta^s\notin C^1\) up to the boundary.

On the other hand, its boundary trace is \(A_s>0\), and therefore
\[
(\partial_\eta)^s u=-A_s<0.
\]
Finally, if \(Q\in\partial B_1\), \(\tau\perp Q\), and \(\nu_Q=-Q\), then
\[
\bigl|Q+t(\nu_Q+\tau)\bigr|^2
=
(1-t)^2+t^2
=
\bigl|Q+t(\nu_Q-\tau)\bigr|^2.
\]
Since both \(u\) and \(\delta\) are radial, the two values appearing in
\eqref{eq:TC} coincide for every sufficiently small \(t>0\). Thus
\eqref{eq:TC} holds with exact cancellation, whereas
\(u/\delta^s\notin C^1\).

\end{document}